\documentclass[12pt]{amsart}
\usepackage{amssymb,epic,eepic,epsfig,amsbsy,amsmath,amscd,color,breqn,float,enumitem}
\usepackage{graphicx}
\usepackage{texdraw}
\usepackage{url}
\usepackage{mathrsfs}
\usepackage{tikz}
\usetikzlibrary{cd}

\def\printname#1{
        \if\draft y
                \smash{\makebox[0pt]{\hspace{-0.5in}
                        \raisebox{8pt}{\tt\tiny #1}}}
        \fi}
\def\lbl#1{\label{#1}\printname{#1}}

                        \theoremstyle{plain}

 \newcommand{\no}[1]{}

\newtheorem{theorem}{Theorem}[section]
\newtheorem{thm}{Theorem}
\newtheorem{lemma}[theorem]{Lemma}
\newtheorem{cor}[theorem]{Corollary}
\newtheorem{prop}[theorem]{Proposition}
\newtheorem{conjecture}{Conjecture}

\theoremstyle{definition}

\newtheorem{remark}[theorem]{Remark}

\newcommand{\lcr}{\raisebox{-5pt}{\mbox{}\hspace{1pt}
                  \includegraphics{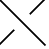}\hspace{1pt}\mbox{}}}
\newcommand{\ift}{\raisebox{-5pt}{\mbox{}\hspace{1pt}
                  \includegraphics{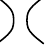}\hspace{1pt}\mbox{}}}
\newcommand{\zer}{\raisebox{-5pt}{\mbox{}\hspace{1pt}
    \includegraphics{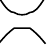}\hspace{1pt}\mbox{}}}

\newcommand{\cC}{\mathcal{C}}
\newcommand{\bm}{\mathbf{m}}

  \def\tKF{\tilde{K}_{\zeta}(F)}
  \def\tZF{\tilde{Z}_{\zeta}(F)}

\def\cD{\mathcal D}
\def\BC{\mathbb C}
\def\BN{\mathbb N}
\def\BZ{\mathbb Z}
\def\BR{\mathbb R}
\def\BQ{\mathbb Q}

\def\cA{\mathcal A}
\def\cB{\mathcal B}

\def\CP{\mathcal P}

\def\Cz{\cC_\zeta}

\def\Sz{\cS_\zeta}

\def\rk{\mathrm{rk}}

\def\cS{\mathscr S}
\def\ot{\otimes}
\def\oS{\mathring{\Sigma}}

\def\bU{{\overline U}}

\def\bn{\mathbf n}

\def\cP{\mathcal P}

\def\ev{{\mathrm{ev}}}

\def\embed{\hookrightarrow}

\newcommand{\al}{\alpha}
\newcommand{\KF}{K_\zeta(F)}
\newcommand{\ZF}{Z_\zeta(F)}

 \def\bm{\mathbf{m}}

 \def\cC{\mathcal C}

\def\tr{\mathrm{tr}}

\def\Tr{\mathrm{TR}}

\def\al{\alpha}

\def\be { \begin{equation} }
\def\ee { \end{equation} }

\def\bt{ {\mathbf t}}

\def\nc{\newcommand}

                  \nc\FI[2]{\begin{figure}
    \begin{center}\input{#1.pstex_t}\end{center}
    \caption{#2}
    \lbl{#1}
  \end{figure}}
\nc\FIG[3]{\begin{figure}
    \includegraphics[#3]{#1.eps}
    \caption{#2}
    \lbl{fig:#1}
    \end{figure}}
\nc\FF[3]{\begin{figure}
    \includegraphics[#3]{#1.eps}
    \caption{#2}
    \lbl{#1}
    \end{figure}}
    \nc\FIGc[3]{\begin{figure}[htpb]
    \includegraphics[height=#3]{#1.eps}
    \caption{#2}
    \label{fig:#1}
    \end{figure}}

    \nc\FIGh[3]{\begin{figure}[htpb]
    \includegraphics[height=#3]{#1.eps}
    \caption{#2}
    \lbl{fig:#1}
    \end{figure}}

\newcommand{\blue}[1]{{\color{blue}[[  #1 ]]}}

\def\ord{\mathrm{ord}}
\def\tA{\tilde A}
\def\tZ{\tilde Z}

 \def\bF{\bar F}
     \def\cV{\mathcal V}
     \def\tr{\mathrm{tr}}

\def\supp{\mathrm{supp}}

\def\cQ{{\mathcal Q}}

\title[Skein Algebra ]{Dimension and Trace of the Kauffman Bracket Skein Algebra}

\author[C. Frohman]{Charles Frohman }
\address{ Department of Mathematics, The University of Iowa}
\email{charles-frohman@uiowa.edu}
\author[J. Kania-Bartoszynska] {Joanna Kania-Bartoszynska}
\address{ Division of Mathematical Sciences, The National Science Foundation}
\email{jkaniaba@nsf.gov}

\author[T. L\^e]{ Thang L\^{e}}
\address{ Department of Mathematics, Georgia Tech}
\email{letu@math.gatech.edu}

\begin{document}

\begin{abstract} Let $F$ be a finite type surface and $\zeta$  a complex root of unity.  The Kauffman bracket skein algebra $\KF$ is an important object in both classical and quantum topology as it has relations to the character variety, the Teichm\"uller space, the Jones polynomial, and the Witten-Reshetikhin-Turaev Topological Quantum Field Theories. We compute the rank and trace of $\KF$ over its center, and we extend a theorem of \cite{FK} which says the skein algebra has a splitting coming from two pants decompositions of $F$.  \end{abstract}

\maketitle

\section{Introduction}\label{Intro}   Let $F$ be a finite type surface and $\zeta$  a complex root of unity.  The Kauffman bracket skein algebra $\KF$ is an important object in both classical and quantum topology as it has relations to the character variety, the Teichm\"uller space, the Jones polynomial, and the Witten-Reshetikhin-Turaev Topological Quantum Field Theories. We recall the definition of $\KF$ in Section \ref{prel}.

The linear representations of $\KF$ play an important role in hyperbolic Topological Quantum Field Theories. In \cite{FKL} we prove the Unicity Conjecture of Bonahon and Wong \cite{BW4} which among other things states that generically all irreducible representations of $\KF$ have the same dimension equal to the square root of the dimension of $\KF$ over its center $\ZF$. Here if $A$ is an algebra whose center is a domain $Z$ then the dimension of $A$ over $Z$, denoted by $\dim_ZA$, is defined to be the dimension of the vector space $A \ot_Z \tZ$ over the field of fractions $\tZ$ of $Z$. The calculation of the dimension $\dim_\ZF \KF$ is one of the main result of this paper.
\begin{thm}
 [ See Theorem \ref{thm.dim}]  \label{thm.1}
 Suppose F is a finite type surface of genus $g$ with $p$ punctures and negative Euler characteristic, and $\zeta$ is a root of unity of order $\ord(\zeta)$. Let $m$ be the order $\zeta^4$, then
\[\dim_{\ZF} \KF=  \begin{cases}
m^{6g-6 + 2p} \qquad &\text{if }\  \ord (\zeta) \not \equiv 0 \pmod 4  \\  2^{2g} \, m^{6g-6 + 2p} \qquad &\text{if } \ \ord (\zeta)  \equiv 0 \pmod 4.\end{cases}\]
\end{thm}

We show that  $\tKF:= \KF \ot_\ZF \tZF$, where $\tZF$ is the field of fractions of $\ZF$,  is a division algebra having finite dimension over its center $\tZF$. Thus every element $\alpha \in \tKF$ lies in a finite field extension of $\tZF$ and hence has the reduced trace $\tr_\tZF(\al) \in \tZF$. We recall the definition of the reduced trace in Section \ref{divalg}. The second goal of the paper is to compute the reduced trace of elements of $\KF$.

To state the theorem, denote  by $\cS$ the set of all isotopy classes of simple diagrams on $F$, where a simple diagram is the union of disjoint, non-trivial simple closed curves on $F$.  For each $\al\in \cS$ one can define an element $T(\al) \in \KF$, such that the set $\{ T(\al) \mid \al \in \cS\}$ is a $\BC$-basis of $\KF$ and $T(\al)$ is central if and only if $\al$ is in a certain subset $\Sz$ of $\cS$. See Section \ref{sec.Cheb} for details. The definition of $T(\al)$ involves Bonahon and Wong's threading map \cite{BW2}.
As the $\BC$-vector space $\KF$ has basis $\{ T(\al) \mid \al \in \cS\}$, hence it is enough to compute the trace of each $T(\al)$.
\begin{thm}  [See Theorem \ref{thm.trace}] \label{thm.2}
 Let $F$ be a finite type surface and $\zeta$ be a root of 1. For $\al\in \cS$ one has
$$
\tr_\tZF(T(\al))= \begin{cases}
T(\al) \quad &\text{if $T(\al)$ is central, i.e., }  \al\in \Sz \\
0     & \text{otherwise.}
\end{cases}
$$
\end{thm} 

 Along the way we develop tools for determining when a collection of skeins forms a basis for $\tKF$.
 
The last goal of the paper is to prove that  there exists a splitting of $\tKF$ over its center coming from pairs of pants decompositions of the surface.

\begin{thm}
[See Theorem \ref{r.decomp}]  \label{thm.3}
 Let $F$ be a finite type surface of negative Euler characteristic.
There exist two pants decompositions $\cP$ and $\cQ$ of $F$ such that for any root of unity $\zeta$  the $\tZF$-linear map
 $$
  \psi:\Cz(\cP) \otimes _{\tZF} \Cz(\cQ) \to \tKF, \quad \psi(x\otimes y) \to xy, 
 $$
 is a $\tZF$-linear isomorphism of vector spaces. Here $\Cz(\cP)$ (respectively $\Cz(\cQ)$) is the $\tZF$-subalgebra of $\tKF$ generated by the curves in $\cP$ (respectively in $\cQ$). Both $\Cz(\cP)$ and $\Cz(\cQ)$ are maximal commutative subalgebras of the division algebra $\tKF$.
 \end{thm}

This theorem has an application in defining invariants of links in 3-manifold which will be investigated in a future work.


The paper is organized as follows. In Section \ref{divalg} we survey results about division algebras that have finite rank over their center, and facts about trace, and filtrations of algebras, with a goal of applying these to the Kauffman bracket skein algebra. 
We follow by introducing the Kauffman bracket skein algebra in Section \ref{prel}. Its basis is given in terms of simple diagrams, so we describe ways of parametrizing simple diagrams on a surface.  We also introduce a residue group. In section \ref{stabledt} we show that after enough twisting the Dehn Thurston coordinates of a simple diagram on a closed surface stabilize to become an affine function of the number of twists. This allows us to define stable Dehn-Thurston coordinates.  In Section \ref{residue} we introduce a degree map and use it to formulate a criterion for independence of a collection of skeins over its center. Section \ref{dimkoverz}  computes the dimension of the Kauffman bracket skein algebra over its center, proving Theorem \ref{thm.1}.  In section \ref{sec.42}  we find bases for commutative subalgebras of $\tKF$  generated by the curves in a primitive non-peripheral diagram on $F$ with coefficients in $\tZF$. In Section \ref{calcoftr} we find a formula for computing the trace, proving Theorem \ref{thm.2}. The paper concludes  in Section \ref{pantedec} which proves the splitting theorem (Theorem \ref{thm.3}).

\subsection{Acknowledgment } The authors thank F. Luo and D. Thurston for their comments.

This material is based upon work supported by and while serving at the National Science Foundation. Any opinion, findings, and conclusions or recommendations expressed in this material are those of the authors and do not necessarily reflect the views of the National Science Foundation. T.L.  thanks the CIMI Excellence Laboratory, Toulouse, France, for inviting him on a Excellence Chair during the period of January -- July 2017 when part of this work
was done.

 \section{Division algebras, trace, filtrations}\label{divalg}
 In this section we  survey some well-known facts about division algebras, trace, and filtrations of algebras that will be used in the paper.

\subsection{Notations and conventions}   Throughout the paper $\BN$, $\BZ$, $\BQ$, $\BR$, $\BC$ denote respectively the set of  natural numbers, integers, rational numbers, real numbers, and  complex numbers. Note that $0 \in \BN$. Let $\BZ_2= \BZ/(2\BZ)$ be the field with 2 elements.

A complex number $\zeta$ is a {\em root of 1} if there is a positive integer $n$ such that $\zeta^n=1$, and the smallest such positive integer  is called the {\em order } of $\zeta$, denoted by $\ord(\zeta)$.

All rings are assumed to be associative with unit, and ring homomorphisms preserve 1. A {\em domain} is a ring $A$, not necessarily  commutative, such that if $xy=0$ with $x,y\in A$, then $x=0$ or $y=0$.
For a ring $A$ denote by $A^*$ the set of all non-zero elements in $A$. For example, $\BC^*$ is the set of all non-zero complex number. 
 
 \subsection{Algebras finitely generated over their centers}


 The following is well-known, and we present a simple proof for completeness.

\begin{prop} \label{r.sfield} (a) If $k$ is a field and $A$ is $k$-algebra which is a domain and has finite dimension over $k$, then $A$ is a division algebra.

(b) Let $Z$ be the center of a domain $A$ and $\tZ$ be the field of fractions of $Z$. Assume  $A$ is finitely generated as a $Z$-module.  Then $\tA= A\otimes_ Z \tZ$ is a division algebra.

\end{prop}
\begin{proof}  (a) Suppose $0\neq a\in A$.  The $k$-subalgebra of $A$ generated by $a$,  being a commutative finite domain extension  of the field $k$, is a field. Hence $a$ has an inverse.

(b) Every element of $\tA$ can be presented in the form $a z^{-1}$ where $a\in A$ and $0\neq z \in Z$. From here it is easy to show that $\tA$ is a domain and the natural map $A \to \tA$ is an embedding. Since $A$ is $Z$-finitely generated, $\tA$ is also $\tZ$-finitely generated, and (b) follows from (a). \end{proof}

The above proposition reduces many problems concerning domains which are finitely generated as  modules over their centers to the case of division algebras finitely generated over their centers.
\subsection{Trace}
Suppose $k$ is a field and $A$ is a $k$-algebra which is finite-dimensional as a $k$-vector space. For $a\in A$, the left multiplication by $a$ is a $k$-linear operator acting on $A$, and its trace is denoted by $\Tr_{A/k}(a)$.
The {\em reduced trace} is defined by
  $$ \tr_{A/k} (a)= \frac{1} {\dim_kA}  \Tr_{A/k}(a) \in k.$$

Again, the following is well-known.
\begin{prop}\label{r.trace1} Suppose that $k$ is a field, and $A$ is  division $k$-algebra 
having finite dimension over $k$. Suppose $0\neq a\in A$.

(a)   If  $P(x)=x^l + c_{l-1} x^{l-1} + \dots c_1 x + c_0$ is the minimal polynomial of $a$ over $k$, then $\tr_{A/k}(a)=- c_{l-1}/l$.

(b) If $B$ is a division algebra with $a \in B \subset A$, then $\tr_{B/k}(a)= \tr_{A/k}(a)$.

(c) The function $\tr_{A/k}: A \to k$ is non-degenerate in the sense that for $0\neq a\in A$ there exists $b\in A$ such that $\tr_{A/k}(ab) \neq 0$. In particular, $A$ is a Frobenius algebra.
\end{prop}
\begin{proof} (a)  Let $C$ be the $k$-subalgebra of $A$ generated by $a$,  then  $C$ is  a  field. By the definition of the minimal polynomial,
\begin{equation}C\cong k[x]/(P(x)).\end{equation}

As a $k$-vector space $C$ has basis $\{1,a, \dots, a^{l-1}\}$. It follows that $\Tr_{C/k} = - c_{l-1}$.
     Let $\{a_1,\dots, a_t\}$ be a basis of $A$ over $C$ so that
$ A =\bigoplus_{i=1}^t C a_i$.  Each $C a_i$ is invariant under the left multiplication by $a$, and the action of $a$ on each $Ca_i$ has trace equal to 
$ \Tr_{C/k}(a)$.  Hence,
$$ \Tr_{A/k}(a) = t \Tr_{C/k}(a) = - t   c_{l-1} = (\dim_kA) \frac{-c_{l-1}}{l}.$$
From here we have $\tr_{A/k}(a) = \frac{-c_{l-1}}{l}$.

(b) follows immediately from (a).

(c) Let $b= a^{-1}$. Then $\tr(a b) = \tr(1)=1\neq 0$.
\end{proof}

 \subsection{Maximal commutative subalgebras}  Suppose $A$ is a division algebra with center $k$.
If $C\subset A$ is a maximal commutative  subalgebra, then $C$ is a field and
\be
\dim_kA  = (\dim_k C)^2. \ee

\subsection{Dimension}   Suppose a $\BC$-algebra $A$ has center $Z$ which is a commutative domain. Let $\tZ$ be the field of fractions of $Z$. The {\em dimension } $\dim_ZA$ is defined to be the dimension of the $\tZ$-vector space $A \ot_Z \tZ$.


A {\em filtration compatible with the product} of $A$ is a sequence $\{ F_i\}_{i=0}^\infty$ of $\BC$-subspaces of $A$ such that $F_i \subset F_{i+1}$, $\bigcup_{i=0}^\infty F_i = A$, and $ F_j F_j \subset F_{i+j}$.
 For any subset $X \subset A$ let  $F_i(X) = F_i \cap X$.
 \begin{lemma} \label{r.lim2}
 Suppose $\dim_ZA < \infty$ and $\dim_\BC F_i(A) < \infty$ for every $i$.  There exists a positive integer $u$ such that for all $k \ge u$,
\be \label{eq.lim2}
\dim_{Z} A \le  \frac{\dim_\BC F_k(A)}{\dim_\BC F_{k-u}(Z)}.
\ee
\end{lemma}
\begin{proof}
 Assume $a_1,\dots, a_d \in A$ form a basis of $A\ot_Z \tZ$ over $\tZ$. Let $u$ be a number such that all $a_j$ are in $F_u(A)$. Since $a_1,\dots, a_d$ are linearly independent over $Z$, the sum $\sum_{j=1}^d Z a_j$ is a direct sum. We have
\be  \bigoplus_{j=1}^d F_{k-u}(Z) a_j \subset F_k( \bigoplus_{j=1}^d Z a_j)  \subset F_k(A). 
\label{eq.s9}
\ee
The dimension of the first space in \eqref{eq.s9} is $d \dim_\BC F_{k-u}(Z)$, while the dimension of the last one is  $\dim_\BC F_k(A)$. Hence, 
$ d \dim_\BC F_{k-u}(Z) \le \dim_\BC F_k(A)$, which is \eqref{eq.lim2}.
\end{proof}

\def\vol{\mathrm{vol}}

\subsection{Lattice points in polytope} Suppose $\BR^n$ is the standard $n$-dimensional Euclidean space. A {\em lattice} $\Lambda \subset \BR^n$ is any abelian subgroup of maximal rank $n$. A   {\em convex polyhedron} $Q$ is  the convex hull of a finite number of points in $\BR^n$ and its $n$-dimensional volume is denoted by $\vol(Q)$. Let $ k Q :=\{ kx \mid x \in Q\}.$

\begin{lemma} \label{r.lim3}
 Suppose $\Gamma \subset \Lambda$ are lattices in $\BR^n$ and $Q\subset \BR^n$ is the union of a finite number of  convex polyhedra with $\vol(Q) >0$. Let $u$ be a positive integer. One has
\begin{align}
\lim_{k\to \infty} \frac{ |\Lambda \cap kQ| }{ |\Gamma \cap (k-u)Q|} &= [\Lambda : \Gamma]. \label{eq.vol3}
\end{align}

\end{lemma}
\begin{proof} Define $\vol(\Lambda)$ to be the $n$-dimensional volume of the parallelepiped spanned by a $\BZ$-basis of $\Lambda$. One has
$$
[\Lambda : \Gamma]  = \frac{\vol(\Gamma)}{\vol(\Lambda)},  \qquad
\lim_{k\to \infty} \frac{|\Lambda \cap k Q|}{k^n} = \frac{\vol(Q)}{\vol (\Lambda)} 
$$
from which one easily obtains  \eqref{eq.vol3}.
\end{proof}

\section{Kauffman bracket skein algebra}\label{prel}  The Kauffman bracket skein module of a 3-manifold was introduced independently by Przytycki and Turaev. For a surface the skein module has an algebra structure first considered in  \cite{Turaev}.
In this section we recall the definition of  the Kauffman bracket skein algebra of a finite type surface $F$, and present some results concerning its center. We also explain how to coordinatize the set of curves on $F$ and use coordinates to define a residue group associated to $F$ and a root of 1.

\subsection{Finite type surface} 
An oriented surface $F$ of the form $F= \bF \setminus \cV$, where $\bF$ is an oriented closed connected surface and $\cV$ is finite (possibly empty), is called a {\em finite type surface}. A point in $\cV$ is called a puncture.
 The genus $g= g(F)$ and the puncture number $p= |\cV|$ totally determine the diffeomorphism class of $F$, and for this reason we denote $F= F_{g,p}$. The Euler characteristic of $F$ is $2-2g -p$, which is non-negative only in 4 cases:
 $$ (g,p) = (0,0), (0,1), (0,2), \ \text{or} \ (1,0).$$
 Since the analysis of these four surfaces is simple and requires other techniques, very often we consider these cases separately.

 \def\SS{\mathring{\cS}}
 \def\Se{\cS^\ev}
 \def\Sd{\cS^\partial}
 
 Throughout this section we fix a finite type surface $F=F_{g,p}$.

 In this paper a {\em loop} on $F$ is a unoriented  submanifold diffeomorphic to the standard circle. A loop is {\em trivial} if it bounds a disk in $F$; it is {\em peripheral} if it bounds a disk in $\bF$ which contains exactly one puncture. A {\em simple diagram} is the union of several disjoint non-trivial loops.
 A simple diagram is {\em peripheral} if all its components are peripheral.
 
 If $a: [0,1]\to \bF$ is a smooth map such that $a(0), a(1) \in \cV$ and $a$ embeds $(0,1)$ into $F$ then the image of $(0,1)$  is called an {\em ideal arc}. Isotopies of ideal arcs are always considered in the class of ideal arcs.
 
 Suppose $\al\subset F $ is either an ideal arc or a simple diagram and $\beta \subset F$ is a simple diagram. The geometric intersection number $I(\al,\beta)$ is the minimum of $|\al' \cap \beta''|$, with all possible $\al'$ isotopic to $\al$ and $\beta'$ isotopic to $\beta$. We say that $\al$ is {\em $\beta$-taut}, or  $\al$ and $\beta$ are {\em taut}, if they are transverse and $|\al \cap \beta| = I(\al,\beta)$.
 
 A simple diagram $\al\subset F$ is {\em even} if $I(\al,a)$ is even for every loop $a\subset F$. It is easy to see that $\al$ is even if and only if $\al$ represents the zero element in the homology group $H_1(\bF, \BZ_2)$.

Very often we identify a simple diagram with its isotopy class.
 Denote by $\cS= \cS(F)$ the set of all isotopy classes of simple diagrams on $F$. Let $\Se\subset \cS$ be the subset of all classes of even simple diagrams, and $\Sd\subset \cS$ be the subset of all peripheral ones. For convenience, we make the convention that the empty set $\emptyset$ is a peripheral simple diagram. Thus $\emptyset \in \Sd \subset \Se \subset \cS$.

\subsection{Kauffman bracket skein algebra} 

A {\em framed link}  in  $F\times [0,1]$ is an embedding of  a disjoint union of oriented annuli in  $F\times [0,1]$.  By convention the empty set is considered as a framed link with 0 components and is isotopic only to itself.

For a  non-zero complex number $\zeta$, the {\em Kauffman bracket skein module} of $F$ at $\zeta$, denoted by $\KF$, is the $\BC$-vector space freely spanned by all isotopy classes of framed links in $F \times [0,1]$ subject to the following {\em skein relations}
\begin{align}
\lcr &= \zeta \zer+ \zeta ^{-1}\ift   \label{KBSR}\\
\bigcirc \sqcup L  &=  (- \zeta^2 -\zeta^{-2}) L\label{KBSR2}.
\end{align}
Here the framed links in
each expression are identical outside the balls pictured in the diagrams, and the arcs in the pictures are supposed to have blackboard framing.
If the two arcs in the crossing belong to the same component then it is assumed that the same side of the annulus is up.

\begin{theorem} \cite{PS,SW}  \label{r.basis}
 The set $\cS$ of isotopy classes of simple diagrams is a basis of $K_{\zeta}(F)$  over $\mathbb{C}$.
\end{theorem}

For two framed links $L_1$ and $L_2$ in $S \times [0,1]$, their product, $L_1L_2$, is defined by first isotoping $L_1$ into $F \times (1/2, 1)$ and $L_2$ into $F\times (0, 1/2)$ and then taking the union of the two. This product gives $\KF$ the structure of  a  $\BC$-algebra, which is in most cases non-commutative.
Let  $\ZF$ be the center of $\KF$.

\begin{theorem} Let $F$ be a finite type surface and  $\zeta$  a root of 1.

(a) \cite{PS1} The algebra $\KF$ is a domain.

(b) \cite{FKL} The module $\KF$ is finitely generated as a $\ZF$-module. 
\end{theorem}

\no { 
  
 From Theorem \ref{r.basis} we immediately get  the following statement. 
\begin{prop}\label{transcendent}
Let $\{C_i\}_{i=1}^k$ be a collection of disjoint non-trivial loops on $F$ such that no two of them are isotopic. Then the $\BC$-subalgebra of $\KF$ generated by the $\{C_i\}_{i=1}^k$ is naturally isomorphic to the algebra $\BC[C_1,\dots, C_k]$ of polynomials in $k$ variables $C_1,\dots, C_k$.  \end{prop} 
}

Let $\tKF= \KF \ot_{\ZF} \tZF$, where $\tZF$ is the field of fractions of $\ZF$.
From Proposition \ref{r.sfield}, we have the following corollary.
\begin{cor} \label{r.Fro}
  For a finite type surface $F$ and a root of unity  $\zeta$, the localized algebra $\tKF$ is a division algebra and a Frobenius algebra.
  \end{cor}
  
 \begin{remark} Corollary \ref{r.Fro}
  was proved in \cite{AF} for the case when $p\ge 1$ and $\ord(\zeta) \neq 0 \mod 4$, using explicit calculation of the trace. \end{remark}

\subsection{Chebyshev basis and center}  \label{sec.Cheb}

   The {\em Chebyshev polynomials of the first kind} are defined recursively by
\begin{equation}
 T_0(x)=2 ,\ \  T_1(x)=x\ \mbox{ and} \  T_k(x)=xT_{k-1}(x)-T_{k-2}(x).
 \end{equation}
They 
satisfy the product to sum formula,
\begin{equation} T_k(x)T_l(x)=T_{k+l}(x)+T_{|k-l|}(x).
\label{eq.Che1}
\end{equation}

Suppose $\al\in \cS$ is a simple diagram. Some components of $\al$ may be isotopic to each other. Let 
$C_1, \dots, C_k$ be a maximal collection of components of $\al$ such that no two of them are isotopic. Then there are positive integers $(l_1,\ldots,l_k)$ such that $\al$ is the union of $l_j$ parallel copies of $C_j$ with $j=1,\dots, k$. In other words, 
$  \al = \prod_{j=1}^k C_j^{l_j} \ \text{ in $\KF$}.
$
 Let
\be 
T(\al) = \prod_{j=1}^k T_{l_j} (C_j) \in \KF.
\ee

 \def\Ao{{\cA}^{\bullet}}
 \def\Ad{{\cA}^\partial}
 \def\bA{\bar A}
 \def\bB{\overline {\cB}}
 \def\Ae{{\cA}^\ev}
  \def\Se{{\cS}^\ev}
 \def\bA{\overline{ \cA}}
 \def\bAe{\overline{ \Ae}}
\def\bAd{\overline{ \Ad}}
 \def\Sd{\cS^\partial}
 \def\vol{\mathrm{vol}}

If $\al, \beta\in \cS$ then in general $\al \beta \not \in \cS$. However, if  $I(\al, \beta) =0$ where $\al, \beta \in \cS$, then $\al$ and $\beta$ can be represented by disjoint simple diagrams, and hence $\al \beta \in \cS$. In particular,  
 if $\al \in \cS$ and $ \beta \in \Sd$, then $\al \beta =\beta \al \in \cS$. Further, 
 if $\al\in \cS$, then  $\al^k \in \cS$. 
 
 \no{For a subset $U\subset \cS$ define
\begin{align*}
U^{\star k} &= \{ \al^k \mid \al \in U\}, \\
\Sd \star U &= \{ \beta \al \mid \beta \in \Sd, \al \in U.\}
\end{align*}
}

 Assume that  $\zeta$ is a root of unity. Let $m= \ord(\zeta^4)$. Note that $m = \ord(\zeta)/\gcd(4, \ord (\zeta))$.  Define the following subset $\cS_\zeta$ of $\cS$ by
\begin{itemize}
\item If $\ord(\zeta) \neq 0 \mod 4$, then $\cS_\zeta = \{ \al \beta^m \mid \al \in \Sd, \beta \in \cS\}$.
\item If $\ord(\zeta) = 0 \mod 4$, then $\cS_\zeta = \{ \al \beta^m \mid \al \in \Sd, \beta \in \Se\}$.
\end{itemize} 

 \no{
 
 \be 
 \cS_\zeta := 
 \begin{cases}
\{ \al \beta^m \mid \al \in \Sd, \beta \in \cS\} 
 \quad &\text{if} \ \ord(\zeta) \neq 0 \mod 4 \\
 \{ \al \beta^m \mid \al \in \Sd, \beta \in \Se\}
&\text{if} \ \ord(\zeta) = 0 \mod 4. \end{cases}  
 \ee
}

\begin{theorem} \cite{FKL}\label{center} 
Let $F$ be a finite type surface and $\zeta$ be a root of 1. Recall that $\ZF$ is the center of the skein algebra $\KF$.

(a) $ \{T(\al) \mid \al \in \cS\}$ is a  basis of the $\BC$-vector space  $\KF$.

(b) $\{ T(\al) \mid  \al \in \cS_\zeta\}$ is a  basis of the $\BC$-vector space  $\ZF$.
\end{theorem}

The point is while $\KF$ has a $\BC$-basis parameterized by $\cS$, its center $\ZF$ has a basis parameterized by $\cS_\zeta$.

\begin{remark} The Chebyshev basis is important in the theory of quantum cluster algebras and quantum Teichm\"uller spaces of surfaces.
 It was first used, for the case when $F$ is a torus, by Frohman-Gelca \cite{FG}. The famous positivity conjecture states that the Chebyshev basis is positive \cite{Le, Th}. \end{remark}
 
 \def\fA{{\mathfrak A}}
 
 \subsection{Filtrations on skein algebras} \label{sec.filter}
Suppose $\fA=\{ a_1, \dots, a_k\}$, where each $a_i$  is an ideal arc or a loop on $F$. 

For each $n\in \BN$ define $F_n = F_n^\fA(\KF)$ to be the $\BC$-subspace of $\KF$ spanned by $\al\in \cS$ such that $\sum_{i=1}^k I(\al, a_i) \le n$. The following is well-known and its variants were used extensively in the study of skein algebras, see e.g., \cite{FKL, Le0,Le1}.
\begin{prop} \label{r.filt}
The sequence $\{ F_n\}_{n=0}^\infty$ is a filtration of $\KF$ compatible with the product.
\end{prop}
\begin{proof} It is clear that $F_n\subset F_{n+1}$ and $\bigcup F_n = \KF$. It remains to show that $F_n F_l \subset F_{n+l}$. 
Suppose $\al, \beta\in \cS$ and $\al\in F_n, \beta \in F_l$. The product $\al\beta$ is obtained by placing $\al$ above $\beta$. 
Using the skein relation \eqref{KBSR} we see that $\al\beta= \sum c_\gamma \gamma$, where each $\gamma$ is a diagram obtained by a smoothing of all the crossings in $\al\beta$
 and hence $I(\gamma,a_i) \le I(\al, a_i) + I(\beta,a_i)$. It follows that $\alpha \beta \in F_{n+l}$.
\end{proof}

  \def\Az{\cA_\zeta}
\subsection{Coordinates and residues, open surface case} When $F= F_{g,p}$ has negative Euler characteristic, one can parameterize the set $\cS$ of simple diagrams on $F$ by embedding it into the free abelian group $\BZ^r$, where $r= 6g-6+3p$. This embedding depends on an object that we call the {\em coordinate datum}. In this subsection we describe this embedding  for open surfaces.

Suppose $F=F_{g,p}$ with $p \ge 1$ and  $F$ has negative Euler characteristic, $\chi(F)=2-2g-p$. 

By definition, a {\em coordinate datum} of $F$ is an {\em ordered ideal triangulation}, which is any sequence 
 $\{ e_1, \dots , e_r \}$ of disjoint ideal arcs on $F$ such that  no two  are isotopic. Here  $r= 6g-6+3p=-3\chi(F)$. Such  ideal triangulations always exists and we fix one of them.  The ideal arcs $(e_i)_{i=1}^r$ cut $F$ into triangles. 

 Recall that $I(\al, e_i)$ is the geometric intersection number. Define
$$
\nu : \cS \to \BN^r, \quad \text{by} \quad \nu (\al) := (\nu_i(\al))_{i=1}^r,\ \mbox{where}\ \nu_i(\al) : = I(\al, e_i).
$$

It is known that $\nu $ is injective, and its image 
 $\cA:= \nu (\cS)$ consists of all $(n_1,\dots, n_r) \in \BN^r$ such that 
 \be \label{eq.cA}
\text{whenever $e_i, e_j, e_k$ are edges of a triangle,  $n_i+ n_j + n_k$ is even and $n_i \le n_j + n_k$.}
\ee 
  We call $\nu (\al)$ the {\em edge-coordinates} of $\al$ with respect to the coordinate datum. Let $S: \cA \to \cS$ be the inverse of $\nu $. That is, $S(\bn)$ is the simple diagram whose edge coordinates are $\bn$. Note that $0 = \nu(\emptyset)\in \cA$, and $\cA$ is closed under addition. Hence $\cA$ is a submonoid of $\BZ^r$. For a submonoid $X$ of $\BZ^r$ let $\overline{X}$ be the   subgroup of $\BZ^r$ generated by $X$, then $\overline{X}= \{ x_1 - x_2 \mid x_1, x_2 \in X\}$.
\def\bX{ \overline{X}} 
  \begin{lemma}\label{r.surj}
 Let $X$ be a submonoid of $\BZ^k$ and $Y$ a subgroup of $\overline{X}$. If   $\bX/Y$ is finite then the  monoid homomorphism $\phi: X \to \bX/Y$ is surjective.
\end{lemma}
\begin{proof}
As $X$ generates $\bX$ as a group, $\phi(X)$ generates $\bX/Y$. The monoid $\phi(X)$, being finite, is a group. Hence $\phi(X)= \bX/Y$.
\end{proof}
 \begin{lemma} (a) $\bA$ is the subset of $\BZ^{r}$ consisting of $(n_1,\dots, n_r)$ such that 
 whenever $e_i, e_j, e_k$ are edges of a triangle,  $n_i+ n_j + n_k$ is even.

 (b) The index of $\bA$ in $\BZ^r$ is $2^{4g-5+2p}=2^{-2\chi(F)-1}$.
 \label{r.index1}
 \end{lemma}
\begin{proof} (a) follows from the description \eqref{eq.cA} of $\cA$.

(b) Consider the triangulation as a cellular decomposition of $\bF$ which has $p$ zero-cells and $r$ one-cells.  Identify $(\BZ_2)^r$ with the set of all maps from one-cells to $\BZ_2$, and let $C_1\subset $ be the set of all one-cocycles. Let $f: \BZ^r \to (\BZ_2)^r$ be the reduction modulo 2, then $\bA = f^{-1}(C_1)$. Therefore $\BZ^r / \bA \cong (\BZ_2)^r/C_1$ and hence
\be 
[\BZ^r : \bA] =  |(\BZ_2)^r| /|C_1| = 2^r /|C_1|. \label{eq.eq4}
\ee
The following  sequence is exact
$$ 0 \to H_0(\bF,\BZ_2) \to C_0 \overset \delta \longrightarrow C_1 \to H_1(\bF, \BZ_2) \to 0,$$
where $C_0= (\BZ_2)^p$ is the set of all 0-cochains. 
In an exact sequence of finite groups, the alternating product of orders of groups is 1. Hence 
$$ |C_1| = \frac{|C_0| |H_1(\bF,\BZ_2)|}{|H_0(\bF,\BZ_2)|}= \frac{2^p 2^{2g}}{2} = 2^{2g+p-1}.$$
Plugging this value of $|C_1|$ in \eqref{eq.eq4}, we get the result.
\end{proof}

 The following follows easily from the definition.
 \begin{prop}  \label{r.prod}
 If $\al, \beta\in \cS$ and $I(\al,\beta) =0$, then $\al\beta \in \cS$ and
 \be 
 \nu(\al \beta) = \nu(\al) + \nu(\beta).
 \ee
 \end{prop}

Let $\Ae:= \nu (\cS^\ev), \Ad= \nu (\Sd), \Az = \nu  (\cS_\zeta)$.  Each of  $\cA, \Ae, \Ad, \Az$ is a submonoid of $\BN^r \subset \BZ^r$.


\def\So{\cS^{\bullet}}
\begin{prop}\label{quotmon} (a)  The group $\bAd$ is a direct summand of $\bA$.

(b) The quotient $\bA/\bAe$ is isomorphic to $\mathbb{Z}_2^{2g}$.

\end{prop}

\def\bh{\bar h}
\proof  (a) 
The group $\bAd$ is a direct summand of $\bA$ if and only if it is {\em primitive} in the sense that \\
(*) if $kx \in \bAd$, where $k$ is a positive integer and $ x\in \bA$, 
then $x\in \bAd$. 

Let $\So\subset \cS$ be the set of all simple diagrams containing no peripheral loops, with the convention that the empty diagram is in $\cS^\bullet$, and let $\Ao= \nu(\So)$.  
Every simple diagram can be presented in a unique way as the product of an element in $\Ad$ and an element in $\Ao$. By Proposition \ref{r.prod}  every $x\in \cA$ can be presented uniquely as
\be 
x = x^\bullet + x ^\partial, \quad \text{where} \ x^\bullet \in \Ao, x^\partial \in\Ad.
\label{eq.decomp1}
\ee

Suppose $x\in \bA$ satisfies (*). There are $x_1, x_2 \in \cA$ such that $x = x_1 - x_2$. Since $kx\in \bAd$, there are $y_1, y_2 \in \Ad$ such that $k(x_1- x_2) = y_2 -y_1$. Using the decomposition \eqref{eq.decomp1} for $x_1$ and $x_2$, we get
$$ k x^\bullet_1 + (k x^\partial_1 + y_1) = k x^\bullet_2 + (k x^\partial_2 + y_2).$$
By Proposition \ref{r.prod}, $k x^\bullet_1,  k x^\bullet_2 \in \Ao$.  The uniqueness of  \eqref{eq.decomp1} shows that $k x^\bullet_1 = k x^\bullet_2$, or 
$  x^\bullet_1 =  x^\bullet_2$. Then $x= x_1 - x_2 = x^\partial_1 - x^\partial_2 \in \bAd$, proving (a).

(b) The composition $ \cA \overset {S} \longrightarrow \cS \overset {h} \longrightarrow H_1(\bF,\BZ_2)$, where $h(\al)$ is the homology class of $\al$, is a surjective monoid homomorphism and extends to a surjective group homomorphism $\bh: \bA \to H_1(\bF,\BZ_2)$.  By definition, $\ker \bh = \bAe$. Hence $\bA/\bAe\cong 
H_1(\bF,\BZ_2)\cong \BZ_2^{2g}$.\qed

\def\Res{\mathfrak{R}}
\def\Rz{\mathfrak{R}_\zeta}
\def\bAz{\overline {\Az}}


Suppose $\zeta$ is a root of 1. From  Theorem \ref{center} and the bijection $\nu:\cS \to \cA$, we see that the center $\ZF$ has a $\BC$-basis parameterized by $\Az$ while $\KF$ has a basis parameterized by $\cA$. Hence we want to study the quotient $\cA/\Az$. Define the $\zeta$-residue group
\be \label{resdef} \Rz(F)= \bA/\bAz \ee
which  depends on a coordinate datum of $F$.
Let  $m = \ord(\zeta^4)$. Define the following integer
\be 
D_\zeta(F)= \begin{cases} m^{6g-6+2p} \quad & \text{if} \quad \ord(\zeta) \neq 0 \mod 4 \\
2^{2g} m^{6g-6+2p} &\text{if} \quad \ord(\zeta)  = 0 \mod 4 \end{cases}.
\ee

\def\hA{\hat \cA}
\begin{prop}\label{r.dim1}
 Let $F=F_{g,p}$ have negative Euler characteristic and $p\ge 1$.  Let $\zeta$ be a root of 1. 
 For any coordinate datum of $F$, 
\be 
|\Rz(F)| =  D_\zeta(F).
\ee

 \end{prop} 

 \proof  Case 1: $\ord(\zeta)  \neq 0 \mod 4$. Then $\Az = \Ad + m \cA$, and
 $$ \Rz(F) = \frac{\bA}{\bAd + m \bA} \cong \frac{\bA/\bAd}{ (\bAd + m \bA)/ \bAd } \cong \frac{\bA/\bAd}{ m (\bA/ \bAd) } .$$
 Note that  $\rk \bA = 6g-6 + 3p$  while  $\rk \bAd =p$. Since $\bAd$ is a direct summand of $\bA$ by Proposition \ref{quotmon}, the group $\bA/\bAd$ is free abelian of rank $6g-6+2p$. It follows that $\frac{\bA/\bAd}{ m(\bA/\bAd)}$ has cardinality $m^{6g-6+2p}$.

Case 2: $\ord(\zeta)=   0 \mod 4$. Then $\Az = \Ad + m \Ae$. Since $\Ad \subset \Ae$, the same argument as in Case 1 with $\cA$ replaced by $\cA^\ev$ gives
$$ |\bAe /(\bAd + m \bAe)| =  m^{6g-6+2p}. $$

By Proposition \ref{quotmon}, we have $ \left | \cA / \cA^\ev   \right | = 2^{2g}$. Hence
$$|\Rz(F)|= \left | \frac{\bA}{ \bAd + m \bAe}   \right |   =  \left |  \frac{\bA}{\bAe}    \right | \, \left |    \frac{\bAe}{ \bAd + m \bAe}  \right | = 2^{2g} m^{{6g-6+2p}} .$$
In all cases we have $|  \Rz(F)|= D_\zeta(F)$.
\qed

We use the collection $\fA=\{e_1,\dots,e_r\}$ to define the filtrations $F_k= F_k^\fA(\KF)$, as described in Subsection \ref{sec.filter}. That is, $F_k$ is the $\BC$-subspace spanned by $S(\bn), \bn\in \cA$ such that $|\bn|:= \sum_{i=1}^r n_i \le k$.
\begin{prop} \label{r.prod1}
For $\bn, \bn'\in \cA$ there exists $j(\bn,\bn')\in \BZ$ such that
\be
S(\bn) S(\bn') = \zeta^{ j(\bn,\bn')}  S(\bn+\bn') \pmod F_{|\bn| + |\bn'| -1}.
\ee
\end{prop}
This was proved in \cite{AF,FKL} for a slightly different filtration but the easy proof there works also for this case. In \cite{FKL}, an explicit formula for $j(\bn,\bn')$ is given.

\subsection{Coordinates and residues, closed surface case} \label{sec.Coor2}

Let $F$ be a closed oriented surface of genus $g>1$. A {\em coordinate datum} of $F$  consists of an ordered pants decomposition $\cP$ and a dual graph $\cD$ defined as follows.

{\em An ordered pants decomposition} of $F$ is a sequence
 $\cP=(P_1, \dots, P_{3g-3})$ of disjoint non-trivial loops  such that no two of them are isotopic. The collection $\cP$ cuts $F$ into $2g-2$ pairs of pants (i.e., thrice punctured spheres).
A {\em dual graph $\cD$ to $\cP$} is a trivalent graph embedded into $F$ having exactly $2g-2$ vertices, one in the interior of each pair of pants, and $3g-3$ edges $(e_{i})_{i=1}^{3g-3}$ such that $e_i$ intersects $P_i$ transversally in a single point  and is disjoint with $P_j$ for $j\neq i$.

For technical simplicity we assume that $\cD$ does not have an edge with endpoints in the same vertex, in other words each pair of pants has 3 different boundary components. Such a coordinate datum always exists, and we fix one.
 
 
 Let $N(\cD)$ be a regular neighborhood of $\cD$ and $\Omega=\partial N(\cD)$ be its boundary. We assume  that for each pair of pants $C$ the intersection $C \cap N(\cD)$ is a regular neighborhood (in $C$) of $\cD \cap C$, and $\Omega\cap C$ consists of 3 arcs as in Figure \ref{ppants}. We call  $C \cap N(\cD)$ the red hexagon of $C$.
 
\begin{figure}[H] \begin{center} \scalebox{.33}{\includegraphics{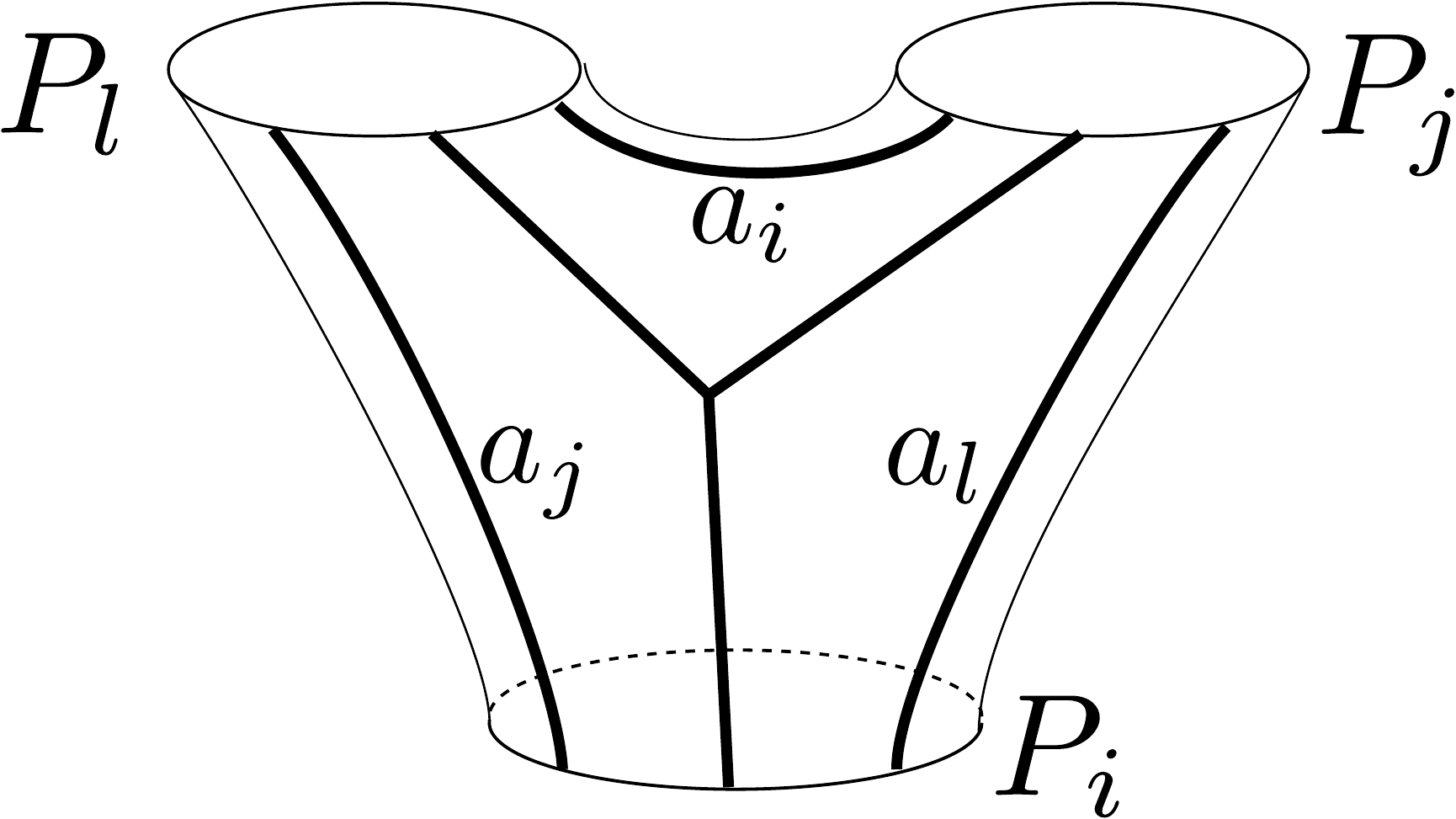}} \end{center}
\caption{The pair of pants $C$ bounded by loops $P_i, P_j, P_l$. The trivalent graph is $\cD\cap C$. The  bold arcs $a_i$, $a_j$, $a_l$ are $\Omega\cap C$. The red hexagon contains the trivalent graph and is bounded by $a_i$, $a_j$, $a_l$ and parts of $P_i, P_j,P_l$.}\label{ppants} \end{figure}

  The curves $\{P_i\}_{i=1}^{3g-3}$ and the system of red hexagons allow to define the Dehn-Thurston coordinates as in \cite{LS}, which is an injective map
$$ \nu: \cS \embed \BZ^{6g-6}, \quad \nu(\al) = (\nu_i(\al))_{i=1}^{6g-g}.$$ 
For $i\le 3g-3$ one has  $\nu_i(\al)=I(\al, P_i)$. The remaining $3g-3$ coordinates of $\nu(\al)$ are the {\em twist} coordinates, where $\nu_{i+3g-3}(\al)$ is the twist coordinate of $\al$ at the loop $P_i$. We use the same conventions for the DT coordinates as in \cite{FKL}.
Our description relates to those of other authors as follows: Our dual graph is embedded as a spine of the complement of the red hexagon of \cite{LS}; in relation to \cite{FLP} it is embedded so that it misses the windows and is disjoint from the triangular model curves. The approach in \cite{Penner} is the same as in \cite{FLP}. All those conventions result in  the same DT coordinates of simple diagrams.

Let $\cA= \nu(\cS)\subset \BZ^{6g-6}$ and let $S: \cA \to \cS$ be the inverse of $\nu$.
The set  $\cA$ consists of all $\bn= (n_1,\dots, n_{6g-6})\in \BZ^{6g-6}$ such that
\be \label{eq.cA2}
\begin{cases}
 \text{  If $ i \le 3g-3$ then $n_i \ge 0$.}\\
\text {If $P_i,P_j,P_l$ bound a pair of pants then $n_i+n_j+n_l$ is even.}\\
\text{ If $n_i=0$ for some $i\le 3g-3$ then $n_{i+ 3g-3}\geq 0$.}
\end{cases}
 \ee
Since these conditions are linear, $\cA$ is a submonoid of $\BZ^{6g-6}$. Let $\bA$ be  the group generated by $\cA$.  

\def\SD{\cS^\Delta}
\def\AD{\cA^\Delta}
\def\cAT{\cA^\Delta}

\def\bode{{\boldsymbol \delta}}
\def\tnu{\tilde \nu}
\begin{lemma} The index of $\bA$ in $\BZ^{6g-6}$ is $2^{2g-3}$.
\label{r.index2}
\end{lemma}
\begin{proof} Identify $\BZ^{6g-6} = \BZ^{3g-3} \oplus \BZ^{3g-3}$ and let $p_1$ be the projection from $\BZ^{6g-6}$ onto the first summand. Note that $\nu(P_j)= \vec{0} \oplus \bode_j$, where $\vec{0} \in \BZ^{3g-3}$ is the zero element and $\bode_j\in \BZ^{3g-3}$ is the element all of whose coordinates are zero except the $j$-th entry which is 1. It follows that $\bA = p_1(\bA) \oplus \BZ^{3g-3}$. Hence the index of $\bA$ in $\BZ^{6g-6}$ is equal to the index of $p_1(\bA)$ in $\mathbb{Z}^{3g-3}$.
By \eqref{eq.cA2}, set $p_1(\bA)$ is the subset of $\BZ^{3g-3}$ such that \\
(*) whenever   $P_i, P_j, P_k$ bound a pair of pants,  $n_i + n_j + n_k$ is even.

The interior  $\mathring N(\cD)$ of $N(\cD)$ is a finite type open surface with  $\chi(\mathring N(\cD))= 1-g$. Let $e_i:= P_i \cap \mathring N(\cD)$, then $(e_i)_{i=1}^{3g-3}$ 
is an ideal triangulation, giving rise to edge-coordinates of simple diagrams on $\mathring N(\cD)$, and the set of  all such edge-coordinates is denoted by $\cA'$. Note that $P_i, P_j, P_k$ bound a pair of pants if and only if $e_i, e_j, e_k$ are edges of an ideal triangle. Condition (*) and  Lemma \ref{r.index1}(a) show that
$ p_1(\bA)= \overline{\cA'}$. Hence
$$ [\BZ^{3g-3}: p_1(\bA)]=  [\BZ^{3g-3}: \overline{\cA'}]= 2^{-2\chi(\mathring N(\cD)) -1}= 2^{2g-3},$$
where the second identity 
 follows from Lemma \ref{r.index1}(b).
 
\end{proof}

\no{
\begin{proof}
Let $\overline{\cA}_0$ denote the subgroup of $\overline{\cA}$ consisting of tuples whose twist coordinates are zero. Let $\{\vec{0}\}\times \mathbb{Z}^{3g-3}$ denote the subgroup of $\mathbb{Z}^{6g-6}$ whose first $3g-3$ coordinates are zero. Since 
$\nu( P_i)=\vec{\delta}_{3g-3+i}$ (the vector all of whose coordinates are zero except the $3g-3+i$ entry),  we have $\{\vec{0}\}\times \mathbb{Z}^{3g-3}\leq \overline{\cA}$.  If $\vec{s}=\nu(\alpha)$ for some $\alpha \in \cS$, then $\vec{s}_0=\nu(\alpha_0)\in\cA$, where $\alpha_0$ has the same pants coordinates as $\alpha$ and has the twist coordinates equal to $0$.

Therefore $\overline{\cA}$ is the interior direct sum of $\overline{\cA}_0$ and $\{\vec{0}\}\times \mathbb{Z}^{3g-3}$.  Hence the index of $\overline{\cA}$ in $\mathbb{Z}^{6g-6}$ is equal to the index of $\overline{\cA}_0$ in $\mathbb{Z}^{3g-3}\times \{\vec{0}\}$.
We abuse notation and think of $\overline{\cA}_0$ as a subgroup of $\mathbb{Z}^{3g-3}$.  Denote the elements of the monoid of pants coordinates $\cA_0$.  We call an element of $\cA_0$ triangular if for any pair of pants, its entries corresponding to the boundary curves of the pairs of pants satisfy the triangle inequality.  If $\vec{n}\in \cA_0$ then there exists $k$ such that $\vec{n}+k\vec{2}$ is triangular.  Hence in $\overline{\cA}_0$ we can write $\vec{n}=(\vec{n}+k\vec{2})-k\vec{2}$.  Since both terms are triangular we have shown that $\overline{\cA}_0$ is generated by Dehn-Thurston coordinates that satisfy the triangle inequality.  We denote this set $\cA^{\Delta}_0$.  

Let $\mathring N(\cD)$ be a regular neighborhood of $\cD$ that is transverse to the pants curves and intersects each pants curve in a single arc.  The intersection of $\cP$ with $\mathring N(D)$ is an ideal triangulation. Identifying $\mathbb{Z}^{3g-3}$ coming from the pants coordinate of the curve $P_i$ with the geometric intersection number with the arc $P_i\cap \mathring N(D)$ the set $\cA^{\Delta}_0$ corresponds exactly to the admissible colorings of the ideal triangulation of $\mathring N(D)$ coming from the $P_i$. Therefore the index of $\overline{\cA}_0$ in $\mathbb{Z}^{6g-6}$ is equal to the index of the admissible colorings of the triangulation of $\mathring N(D)$.  Since the Euler characteristic of $\mathring N(D)$ is half the Euler characteristic of $F$, we have the desired result. \qed

Note that this index is determined by the pants coordinates, since $\bA$ has last $3g-3$ components equal to $\BZ$. 
The group $\bA$ is  spanned by the image of triangular diagrams.
The open surface $\mathring N(\cD)$  has Euler characteristic half of $E(F)$ and the formula follows from Lemma \ref{r.index1}.
\end{proof}
}

Suppose $\zeta$ is a root of 1. Let $\Ae= \nu(\Se)$ and $\Az = \nu(\Sz)$, then  $\Ae, \Az$ are submonoids of $\cA$.

 Let $m = \ord(\zeta^4)$.   Define the $\zeta$-residue group $\Rz(F)$ and the number $D_\zeta(F)$ just like in the case of open surfaces (noting $p=0$). That is,
\be \label{resdef1} \Rz(F)= \bA/\bAz \ee
which  depends on a coordinate datum of $F$. The formulas for $D$ are the same:
\be 
D_\zeta(F)= \begin{cases} m^{6g-6} \quad & \text{if} \quad \ord(\zeta) \neq 0 \mod 4 \\
2^{2g} m^{6g-6} &\text{if} \quad \ord(\zeta)  = 0 \mod 4.\end{cases}
\ee

\def\hA{\hat \cA}
\begin{prop}\label{r.dim1a}
 Suppose $F$ is a closed finite type surface with $g(F)\ge 2$.  Let $\zeta$ be a root of 1. 
  For any Dehn-Thurston datum of $F$,
\be 
|\Rz(F)| =  D_\zeta(F).
\ee
 \end{prop} 

 \proof The proof is almost identical to that in the case of open surfaces.
 
 (a) Case 1: $\ord(\zeta)  \neq 0 \mod 4$. In this case $\Az = m\cA$, and
 $ \Rz(F) = \bA/m \bA$. Since
   $\rk \bA = 6g-6$, we have  $|\bA/m \bA| = m^{6g-6}$.

Case 2: $\ord(\zeta) =  0 \mod 4$. In this case $\Az = m \Ae$. 

First note that $|\bA /\bAe| = 2^{2g}$. The proof of this fact is identical to that of Proposition \ref{quotmon}.
\no{In fact, the composition $ \cA \overset {S} \longrightarrow \cS \overset {h} \longrightarrow H_1(\bF,\BZ_2)$, where $h(\al)$ is the homology class of $\al$, is a surjective monoid homomorphism and extends to a surjective group homomorphism $\bh: \bA \to H_1(\bF,\BZ_2)$.  By definition, $\ker \bh = \bAe$. Hence $\bA/\bAe\cong 
H_1(\bF,\BZ_2)\cong \BZ_2^{2g}$, and $|\bA /\bAe| = 2^{2g}$.}
Now we have
$$ |\Rz(F)| = |\bA/m \bAe| = |\bA /\bAe| |\bAe/ m \bAe| = 2^{2g} m^{6g-6}= D_\zeta(F). \ \qed$$

 When $F$ is closed we don't have a nice product formula like the identity in Proposition \eqref {r.prod1}. However, this identity   still holds for the class of triangular simple diagrams defined as follows.
 
 A simple diagram $\al\subset F$ is {\em triangular} with respect to $\cP$ if it is $\cP$-taut  and for every pair of pants $C$ each connected component of $\al \cap C$ is an arc whose two endpoints are in two different components of $\partial C$. In particular, $\al$ cannot have a component isotopic to any  $P_i$.

Let $\SD\subset \cS$ be the subset consisting of triangular simple diagrams, and $\AD= \nu(\SD)$. 
Then $\bn= (n_1,\dots,n_{6g-6})\in \cA$ is in $\cAT$ if and  only if
\begin{itemize}
\item whenever  $P_i, P_j, P_k$ bound a pair of pants,  $n_i \le n_j + n_k$, and
\item whenever $n_i=0$ for some $i\le 3g-3$, one has  $n_{i+3g-3}=0$.
\end{itemize}

We use the collection $\fA=\{P_1,\dots,P_{3g-3}\}$ to define the filtrations $F_k= F_k^\fA(\KF)$, as described in Subsection \ref{sec.filter}. That is, $F_k$ is the $\BC$-subspace spanned by $S(\bn), \bn\in \cA$ such that $|\bn|_1:=\sum_{i=1}^{3g-3} n_i \le k$. Unlike the case of open surface,  $F_k$ has infinite dimension over $\BC$.
\begin{prop}   \cite{FKL}\label{prsi}  For $\bn, \bn'\in \cAT$ there is  $j\in\BZ$ such that
\be
S(\bn) S(\bn') = \zeta^j  S(\bn+\bn') \mod F_{|\bn|_1 + |\bn'|_1 -1}. 
\label{eq.prod2}
\ee
\end{prop}
In \cite{FKL} we proved a stronger result, giving the exact value of $j$ in \eqref{eq.prod2}.

\section{Stable Dehn-Thurston coordinates}\label{stabledt} 

Throughout this section  $F$ is a closed surface with $g=g(F) \ge 2$, and  with a  fixed  coordinate datum $(\cP,\cD)$. 
Proposition \ref{prsi} only works for triangular simple diagrams. We show that after enough twisting by $\Omega$ a simple diagram  becomes triangular and its  DT coordinates  become an affine function of the number of twists. 

\subsection{Stable DT coordinates}
  
\begin{theorem} \label{thm.stab} Let $F$ be a closed surface of genus $g\ge 2$ equipped with  coordinate datum $(\cP,\cD)$. Let $h_{\Omega}:F\rightarrow F$ be the product of the Dehn twists about $o_1, \dots, o_l$, which are the components of 
$\Omega=\partial N(\cD)$.
 
Let $\al\subset F$ be a simple diagram. There exists  $\eta(\al) \in \BZ^{6g-6}$ such that if $k$ is large enough 
then $h^{k}_\Omega(\al)$ is triangular with respect to $\cP$, and
\be 
\nu(h^{k}_\Omega(\al)) = k \mu(\al) + \eta(\al),\label{stabdt}
\ee
where
\be
\mu(\al)= \sum_{j=1}^{l} I(\al,o_j) \nu(o_j).  \label{eq.mu}
\ee
In particular, 
 the last $3g-3$ coordinates of $\mu(\al)$ are equal to $0$.
 
\end{theorem}

Note that the action of any element of the mapping class of the surface on diagrams extends linearly to yield an automorphism of the skein algebra, In specific, it makes sense to talk about $h^{k}_\Omega(x)$ for $x\in \KF$.

We present the proof of Theorem \ref{thm.stab} in Subsection \ref{sec.pf1}.

\subsection{Piecewise affine functions} A function $f: \BR^k \to \BR^l$ is {\em  affine} if there is an $l \times k $ matrix $A$ and a vector $B$ such that 
$$ f(x) = A\cdot x + B.$$
A function $f: \BR^k \to \BR^l$ is {\em piecewise affine} if  
there is  a finite collection  of proper affine subspaces 
of $\BR^k$ such that 
in the closure of any connected component of 
the complement of these affine spaces $f$ is equal to an affine function. It is easy to see that the class  of piecewise affine functions is closed under linear combinations and compositions.

A function $f: X \to \BZ^l$,  where $X\subset \BR^k$, is {\em piecewise affine} if it is the restriction of a  piecewise affine function $\overline{f}: \BR^k \to \BR^l$.

\begin{lemma} If $f: \BN \to \BN$ is convex and bounded from above by an affine function, then $f$ is piecewise affine.
\label{r.convex}
\end{lemma}
\begin{proof}
The lemma  follows easily from the definition. 
\end{proof} 

\begin{lemma} For any $\al, \beta \in \cS$ and $k>0$ one has
\be 
 |I(h^{k}_\Omega(\al), \beta) - k \sum_{j=1}^l I(\al, o_j) I(o_j, \beta) | \le I(\al, \beta).
 \label{eq.b1}
 \ee
\label{r.b1}
\end{lemma}
\begin{proof}
This is  a special case of Proposition A1 of \cite[Section 4]{FLP}.
\end{proof}

In order to prove the next proposition we need to explore the topology of Dehn-Thurston datum $(\cP,\cD)$ of a closed surface $F$. To define a geometric intersection of  a simple diagram $S$ with the graph $\cD$ we add the assumption that $S$ misses the vertices of $\cD$. That is,  $I(S,\cD)$ is the minimum cardinality of $S'\cap \cD$ where $S'$ is isotopic to $S$, misses the vertices of $\cD$ and is transverse to its edges.

\begin{lemma} \label{bound} If $(\cP,\cD)$ is DT-datum for the closed surface $F$ and $S$ is simple diagram with Dehn-Thurston coordinates $(\bn (S), \bt(S))$,
 then \be \sum_j|t_j(S)|\leq 2I(S,\cD).\ee \end{lemma}

\begin{proof} Given Dehn-Thurston datum for $F$ let $A_j$ denote the annuli which are collars of the pants curves $P_j$, and let $Q_i$  be the pairs of pants that are the complement of $\cup A_j$ in $F$. These are the shrunken pairs of pants, versus the pairs of pants $C_j$ as defined in section \ref{sec.Coor2}.
By assumption $\cD$ is transverse and minimizes its intersection with the boundaries of the annuli $A_j$. We say that a simple diagram is in {\em standard position} if
\begin{itemize}
\item its intersection with the $Q_i$ is isotopic to  standard model curves in the complement of $\cD$, 
\item its intersection with $\partial A_j$ is disjoint from $\partial A_j \cap \cD$, 
\item it minimizes its intersection with $\cD\cap A_j$, for each $j$.
\end{itemize}  
If a simple diagram is in standard position then its twists coordinates are given by its signed intersection numbers with 
$\cD\cap A_i$.

We are most interested in the case that $S$ is triangular. In Figure \ref{model}  we show the triangular model curve $d_{12}$ and another curve $d'_{12}$ that will play a role in the following. For each pair of boundary components of each pair of pants $Q_i$ there are two curves like this, the model curve $d_{ij}$ and its mate $d'_{ij}$. Recall that model curves describe  possible ways in which a simple diagram in standard position intersects a pair of pants (see, e.g., \cite{FKL}).


\begin{figure}[H]\begin{center}\scalebox{.5}{\begin{picture}(150,151)\includegraphics{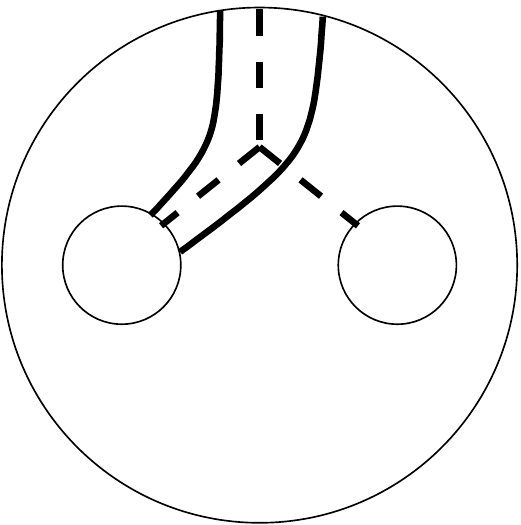}\put(-110,120){$d_{12}$}\put(-50,120){$d'_{12}$}\end{picture}}
\end{center}
\caption{The model curve $d_{12}$ and the curve $d'_{12}$.}\label{model}\end{figure}

We say a triangular diagram $S$ is in {\em special position} if:
\begin{itemize}
  \item It realizes $I(S,\partial A_j)$ for all $j$;
\item It realizes $I(S,\cD)$;
\item It does not intersect $\cD\cap \partial A_j$ for any $j$;
\item It minimizes the intersection  $S\cap \cD \cap Q_i$ for all $i$, among all $S$ satisfying the first condition.
\end{itemize}

It is easy to see that if $S$ is a triangular diagram in special position then it intersects each $Q_j$ in curves parallel to model curves $d_{kl}$ and the new curves $d'_{kl}$.  

To move a curve from special position to standard position, each curve of type $d'_{kl}$ needs to be isotoped to a curve of type $d_{kl}$.  In the Figure \ref{deformed}  we show a curve of the form $d_{kl}'$ in the process of being pushed into standard position. Its intersection with $\cD$ needs to be pushed inside the annuli. As the result the intersection of $S$ with $\cD\cap A_j$ in each of the annuli on either end of $d_{kl}$ is incremented by $1$.

\begin{figure}[H]\begin{center}\scalebox{.5}{\begin{picture}(150,151)\includegraphics{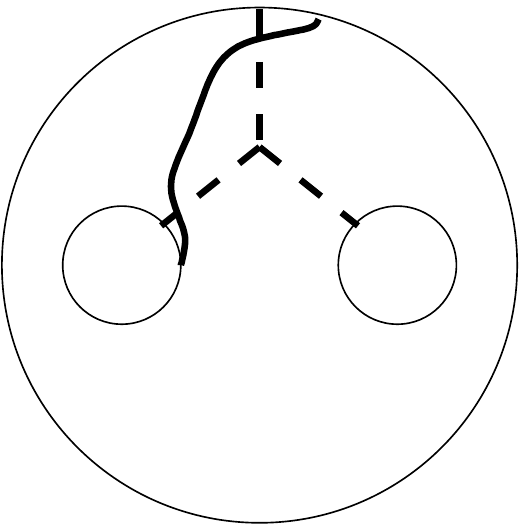}\end{picture}}
\end{center}
\caption{Deforming the  curve $d_{12}'$ on the way to standard position}\label{deformed}\end{figure}

Thus \be\sum_j |t_j(S)|\leq 2I(S,\cD).\ee\end{proof}


\begin{prop} \label{r.fact1} Let $F$ be a closed surface of genus $g$ with fixed coordinate datum.

(a) 
There exists an additional collection of $6g-6$ loops $P_j$ with $j=3g-2, \dots, 9g-9$, and
 a piecewise affine function $G: \BZ^{9g-9} \to  \BZ^{6g-6}$ such that for all $\al\in \cS$,
\be 
\nu(\al) = G\Big( I(\al,P_1), \dots , I(\al, P_{9g-9})  \Big ). \label{eq.twistpl}
\ee

(b) For any  $\al, \beta \in \cS $,  the function $f_{\al,\beta}: \BZ  \to \BZ$, defined by
 $f_{\al,\beta}(k)= I(h_{\Omega}^k(\al), \beta))$, is piecewise affine.

(c)  For any $\al\in \cS$, the twist coordinates of $h_{\Omega}^k(\al)$ are bounded, meaning that  there is a constant $M>0$ such that for all $k$ and all $i=3g-2,\dots, 6g-6$, one has $|(h_{\Omega}^k(S))_i| < M$.
\end{prop}

\begin{proof} (a)  This follows from  \cite[Proposition 4.4]{LS}. In \cite{LS} it was proved that the Dehn-Thurston coordinates can be expressed as  homogeneous continuous functions of the intersection numbers. But the explicit functions appeared there are actually piecewise linear.

(b) By \cite[Corollary 3]{Luo2}, the function $f_{\al,\beta}: \BN \to \BN$ is convex. 
By Lemma \ref{r.b1}, $f_{\al,\beta}$ is bounded from above by an affine function.
\no{
When doing Dehn twist about $\Omega$ we modify $\al$ in a small neighborhood of $\Omega$ by inserting arcs wrapping about $\Omega$. It follows that $I(h_{\Omega}^k(\al), b))$ is bounded from above by an affine function.}
By Lemma \ref{r.convex}  
$f_{\al,\beta}$ is piecewise affine.

(c) Notice that applying $h_{\Omega}^k$ to a diagram $S$ does not increase $i(S,\cD)$. By Lemma \ref{bound}  the twist coordinates of $h_{\Omega}^k(S)$ are bounded above by $2i(S,\cD)$.
\end{proof}

\subsection{Proof of Theorem \ref{thm.stab}} \label{sec.pf1}
Let $\al\in \cS$.
Parts (a) and (b) of Proposition  \ref{r.fact1} along with Formula \eqref{eq.twistpl} imply that the function $k \to \nu(h^k_\Omega(\al))$ is piecewise affine. Hence it is affine for big enough $k$. Thus there exist $k_0$ and $ \mu(\al), \eta(\al) \in \BZ^{6g-6}$ such that if $k \ge k_0$ then
\be 
\nu(h^{k}_\Omega(\al)) = k \mu(\al) + \eta(\al). \label{stabdt1a}
\ee
By Proposition \ref{r.fact1}(c),  the twist coordinates of $h^{k}_\Omega(\al)$ are bounded. It follows that the last $3g-3$ coordinates of $\mu(\al)$ must  be 0.

Assume  $ i\le 3g-3$. By definition,  $\nu_i(h^{k}_\Omega(\al))=I(h^{k}_\Omega(\al), P_i)$.
\no{  There are an upper bound and an lower bound for the intersection numbers of Dehn twists of simple diagrams in \cite[Appendix to Section 4]{FLP}. By \cite[Proposition A1, Section 4]{FLP}, we have
$$ |I(h^{k}_\Omega(\al), P_i) - k \sum_{j=1}^l I(\al, o_j) I(o_j, P_i) | \le I(\al, P_i).$$
}
Comparing the slope of $\nu_i(h^{k}_\Omega(\al))$ in  \eqref{stabdt1a} and \eqref{eq.b1}, with $\beta= P_i$, we get 
$$ \mu_i(\al)= \sum_{j=1}^l I(\al, o_j) I(o_j, P_i) = \sum_{j=1}^l I(\al, o_j) \nu_i(o_j).$$
As the twists coordinates of $o_j$,  as well as the last  $3g-3$ coordinates of $\mu(\al)$, are all 0, we have
\be   \mu(\al) =  \sum_{j=1}^l I(\al, o_j) \nu(o_j).
\label{eq.stab5}
\ee
It remains to show that $h^{k}_\Omega(\al)$ is triangular for large $k$. Suppose $P_i, P_j, P_l$ bound a pair of pants $C$. An arc in $C$  having one endpoint in $P_j$ and one end point in $P_l$ is called a $(j,l)$-arc. An arc having both endpoints in $P_i$ is called an $(i,i)$-arc. 
\begin{lemma} Let $\beta$ be a $\cP$-taut  simple diagram. 

(a) If $\beta\cap C$ has a $(j,l)$-arc, then $\beta\cap C$ does not have $(i,i)$-arcs.

(b) If $\beta\cap C$ does not have $(i,i)$-arcs, then $\nu_i(\beta)\le   \nu_j(\beta)  + \nu_l(\beta)$.
\end{lemma} \label{r.arcs}
\begin{proof}
Both statements follow from the well-known facts \cite{FLP}:\\
$\bullet$  The number of $(j,l)$-arcs is $\max(0, (\nu_j(\beta)  + \nu_l(\beta) - \nu_i(\beta))/2)$,\\
$\bullet$  The number of $(i,i)$-arcs is $\max(0, ( \nu_i(\beta)-\nu_j(\beta)  - \nu_l(\beta) )/2)$.
 \end{proof}
 Note that $\Omega\cap C$ consists of 3 arcs,  a $(j,l)$-arc  $a_i$,  a $(k,i)$-arc $a_j$, and a $(j,i)$-arc $a_k$, see Figure \ref{fig:ppants}. For each $s=i,j,l$ let $m_s$ be the intersection number of $\al$ with the component of $\Omega$ containing $a_s$. By \eqref{stabdt1a} and \eqref{eq.stab5}, for $k \ge k_0$ we have
 \begin{align}
 \nu_i(h^k_\Omega(\al))&= k(m_j + m_l) + \eta_i(\al) \label{eq.b2} \\
 \nu_j(h^k_\Omega(\al))&= k(m_i + m_l) + \eta_j(\al) \notag \\
 \nu_l(h^k_\Omega(\al))&= k(m_i + m_j) + \eta_l(\al). \notag
 \end{align}
 It follows that
 \be \nu_j(h^k_\Omega(\al)) + \nu_l(h^k_\Omega(\al)) - \nu_i(h^k_\Omega(\al)) = 2 k m_i + \eta_l(\al)+ \eta_j(\al) - \eta_i(\al).\ee
 Hence if $m_i >0$ then, for sufficiently large $k$ we have
 \be  
 \nu_j(h^k_\Omega(\al)) + \nu_l(h^k_\Omega(\al)) - \nu_i(h^k_\Omega(\al)) \ge 0. \label{eq.in2}
 \ee

 Suppose now $m_i=0$. Let $o_s$ be the component of $\Omega$ containing $a_i$. Then $I(\al, o_s)=m_i=0$ and we can assume $\al \cap o_s= \emptyset$. Since $o_s$ is a component of $\Omega$, for any $k>0$ we also have $h^k_\Omega(\al) \cap o_s= \emptyset$.
 Let $\beta= h^k_\Omega(\al) \cup o_s$, then $\beta$ has a $(j,l)$-arc in $C$, which is $a_i$. By Lemma \ref{r.arcs}(a), $\beta$ does not have $(i,i)$-arcs. As $h^k_\Omega(\al)$ is a sub-diagram of  $\beta$, it does not have an $(i,i)$ arc neither. By Lemma \ref{r.arcs}(b), we have \eqref{eq.in2}. Thus in all cases, we have \eqref{eq.in2} for large enough $k$.
 
 It remains to show that  for $k$ large, $h^k_\Omega(\al)$ does not have a component isotopic to $P_i$. Suppose $h^{k_1}_\Omega(\al)$ has a component isotopic  to $P_i$ for some $k_1$. As $P_i$ has non-trivial intersection with the components of $\Omega$ which contains $a_j$ and $a_l$, we have $m_j, m_l >0$. From \eqref{eq.b2} it follows for large $k$ we  have $\nu_i(h^k_\Omega(\al)) >0$, implying $h^k_\Omega(\al)$ does not have components isotopic to $P_i$. This completes the proof of Theorem \ref{thm.stab}. \qed

\subsection{More on Theorem \ref{thm.stab}} We  call $\tnu(\al):= (\mu(\al),\eta(\al))\in \BZ^{12g-12}$ the {\em stable DT-coordinate} of $\al\in \cS$ with respect to the coordinate datum $(\cP, \cD)$.
Recall that $\bA$ is the subgroup of $\BZ^{6g-6}$ generated by the monoid $\cA$. Let $\vec 0\in \BZ^{3g-3}$ be the element having all $0$'s as entries, and $\bode_i\in \BZ^{3g-3}$ be the element having all $0$'s as entries except a $1$ in the $i$-th entry.
\begin{prop} \label{r.aden}
(a) The map $\tnu:\cS \to \BZ^{12g-12}$  is injective.

(b) For any $\al\in \cS$ one has $\eta(\al) \in \bA$.

(c) If $I(\al, \Omega)=0$ then $\eta(\al) = \nu(\al)$ and $\mu(\al) =(\vec0, \vec 0)$.

(d) For all $k\in \BZ$ one  has 
\be 
\eta(h^k_\Omega(\al))= \eta(\al) + k \sum_{i=1}^l I(\al, o_i)\nu(o_i). \label{eq.eta}
\ee

(e) One has $\eta(P_j) = (\vec 0, - \bode_j)$.
\end{prop}
\begin{proof} (a) If $(\mu(\al), \eta(\al))= (\mu(\beta), \eta(\beta))$, then \eqref{stabdt} shows $h^k_\Omega(\al)=h^k_\Omega(\beta) $ for large $k$. Applying $h^{-k}_\Omega$, we get $\al=\beta$.

(b) As $\eta(\al)= \nu(h^k_\Omega(\al)) - k \sum_i I(\al, o_i) \nu(i)$ is the difference of two elements in $\cA$, we have $\eta(\al) \in \bA$.

(c) If $I(\al, \Omega)=0$ then $h^k_\Omega(\al)= \al$ for all $k$. The result follows.

(d) Applying $h_\Omega^j$ to \eqref{stabdt}, noting that $h_\Omega(o_i)=o_i$,  we get 
$$\eta(h^j_\Omega(\al))= \eta(\al) + j\sum_{i=1}^l I(\al, o_i)\nu(o_i),$$
 which is \eqref{eq.eta} with $k=j >0$. Replacing $\al$ by $h^{-k}_\Omega$ in \eqref{eq.eta} we get \eqref{eq.eta} with $k$ replaced by $-k$.

(e) 
Note that  $\nu(P_j) = (\vec 0,  \bode_j)$. After twisting once along $\Omega$ its  twist coordinate becomes $-\bode_j$,
\be
\nu(h^k_\Omega(P_j))= (\vec 0, - \bode_j)+ k \sum_{i=1}^l I(P_j, o_i)\nu(o_i)
\ee
Since $h_{\Omega}(P_j)$ makes no bigons with $P_j$ and its intersection with $D$ is contained inside the annulus around $P_j$, therefore its DT coordinates change linearly, and the result follows.
\end{proof}

 \no{In the next section we will develop a similar concept, the stable residue of a simple diagram, which also takes into account the arithmetic of a root of unity.

\begin{prop}[Addendum to Theorem \ref{thm.stab}] Let $S$ be a simple diagram on surface $F$ with Dehn-Thurston data $(\cP, \cD)$. Suppose that $P_j\in \cP$ is one of the pants curves. If there is a component $a$ of $\Omega$ such that $i(S,a) i(a,P_j) \neq 0$,  i.e., $i(S,a) n_j(a)\neq 0$, then $j$-th stable twist coordinate of $S$ is non-positive. That is, the $(3g-3 +j)$-th component of $Y_S$ is non-positive.

On the other hand, if $i(S,a) i(a,P_j) =0 $ for all components $a$ of $\Omega$, then the $3g-3+j$-th component of $Y_s$ is equal to $t_j(S)$. 
\end{prop}

}

\def\KFs{\KF^*}
\def\ZFs{\ZF^*}
\def\dz{\deg_\zeta}

\section{Independence over the center}\label{residue}
We formulate a criterion for independence of a collection of elements of $\KF$ over the center.
Throughout this section we fix a finite type surface $F= F_{g,p}$ with negative Euler characteristic equipped with coordinate datum. One can define the coordinates $\nu: \cS \embed  \BZ^r$, where $r= 6g-6+3p$. The set of possible coordinates $\cA= \nu(\cS)$ is a submonoid of $\BZ^r$. Let $\bA$ denote the subgroup of $\BZ^r$ generated by $\cA$. For a root of unity $\zeta$ we also define the submonoid $\Az$ and its group $\bAz$ as in  Section \ref{prel}.

For a ring $R$ we denote by $R^*$ the set of non-zero elements of $R$. Since $\KF$ is a domain, $\KFs$ is a monoid under multiplication.
\subsection{General result}\label{genres}The degree of  polynomials in one variable  satisfies the following two  properties: for non-zero polynomials $x$ and  $y$,
\begin{itemize}
\item[(i)] $ \deg(xy) = \deg(x) + \deg(y)$ (monoid homomorphism),
\item[(ii)] if $\deg(x_i)$ for $i=1, \dots, d$ are pairwise distinct, then $\sum_i x_i \neq 0$.
\end{itemize}
For a domain $R$, a {\em degree  map} is a map $\deg: R^* \to M$, where $M$ is a monoid, satisfying the above two properties.
\no{ Note that  (iii) is implied by
 the following 
\begin{itemize}
\item[(iii)] if $\deg(x) > \deg(y)$, then $x+y \neq 0$ and $\deg(x+y) = \deg(x)$.
\end{itemize}
}

\begin{theorem}
\label{r.gen1} Given a finite type surface $F= F_{g,p}$ with negative Euler characteristic and a fixed coordinate datum, let $\zeta$ be a root of 1. 
There exists a degree map
 \be\deg: \KFs \to \bA\ee 
 such that 
 $ \deg(\ZFs) \subset \bAz$. Moreover, the composition 
 \be \deg_\zeta: \KFs \overset{ \deg} \longrightarrow  \bA \to  \bA/\bAz= \Rz
 \label{eq.dzdef}
 \ee
  is a surjective monoid homomorphism.
\end{theorem}  
We will construct the map $\deg$ in later subsections. We want to mention an important corollary that we will use in the future.
\begin{cor} \label{r.dimlow}
Assume the hypothesis of Theorem \ref{r.gen1}. 

(a)  If $x_1, \dots, x_d\in \KFs$ such that $\dz(x_1), \dots, \dz(x_d)$ are pairwise distinct, then $x_1, \dots, x_d$ are linearly independent over $\ZF$.

(b) One has $\dim_\ZF \KF \ge |\Rz| = D_\zeta(F)$.
\end{cor}
\begin{proof} (a) 
Suppose $z_1, \dots z_d\in \ZFs$. 
From the assumption, the elements $\deg(z_1 x_1), \dots, \deg(z_d x_d)$ are pairwise distinct in $\bA$. By Property (ii)  of degree maps, the sum $\sum z_i x_i \neq 0$.

(b) Since $\dz( \KFs )= \Rz(F)$, from  (a) we have $\dim_\ZF \KF \ge |\Rz|$, which is equal to $D_\zeta(F)$ by Propositions \ref{r.dim1} and \ref{r.dim1a}.
\end{proof}

As $\dz(\ZF)=0$, the map $\dz$ extends to a surjective group homomorphism, also denoted by $\dz$:
\be 
\dz: \tKF^* \to \bA/\bAz= \Rz.
\label{eq.Rd}
\ee
\subsection{Lead term}
Since $\cS$ is a $\BC$-basis of $\KF$, for every $x\in \KFs$ there is a  unique set $\supp(x) \subset \cS$ such that $x$ has the presentation 
\be 
x = \sum_{\al \in \supp(x)} c_\al \al, \quad 0\neq c_\al \in \BC.
\label{eq.supp}
\ee
If $\le$ is a total order on $\cS$, then $c_\al \al$, where $\al =\max \supp(x)$, is called the $(\le)$-lead term of $x$, and we can write
\be 
x = c_\al \al + G_<(\al),
\ee
where $G_<(\al)$ is the $\BC$-span of $\{ \beta\in \cS \mid \beta < \al\}$.

\subsection{Proof of Theorem \ref{r.gen1} for open surfaces} Suppose $F$ is an open
 finite type surface with negative Euler characteristic and coordinate datum $\{e_1, \dots, e_r\}$.
Let $\iota: \BZ^r \embed \BZ^{r+1}$ be the embedding 
$$ \iota(\bn) = (|\bn|, \bn),$$
where $|\bn| = \sum n_i$. Let $\unlhd$ be the total order on $\cS$ and $\cA$ induced from the  lexicographic order on $\BZ^{r+1}$ via the embeddings
$$ \cS \overset{\nu}\longrightarrow  \cA \embed \bA \embed \BZ^r \overset{\iota}\embed\BZ^{r+1}.$$

The order $\unlhd$ makes $\bA$ an ordered group. Define 
$\deg: \KFs \to \bA$ by 
\be 
   \deg(x) = \nu(\al),\quad  \text{where} \ \al = \max \supp(x).
   \label{eq.defdeg1}
   \ee
Suppose $x_1\dots, x_n \in \KFs$. If $\deg(x_1), \dots, \deg(x_d)$ are distinct, then $\sum_i x_i\neq 0$ and $\deg(\sum_i x_i) = \max \{ \deg(x_i)\}$.
 Proposition \ref{r.prod1} implies that  $\deg(xy)= \deg(x) + \deg(y)$ for any $x,y\in\KF$. Hence $\deg$ is a degree map.
By definition $\deg(\cS) = \cA$, which by Proposition \ref{r.surj} surjects onto $\bA/\bAz=\Rz$. Since $\cS\subset \KFs$, we also have $\dz(\KFs)= \Rz$. This completes the proof of Theorem \ref{r.gen1} for open surface. 

\def\tnu{\tilde \nu}
\subsection{Proof of Theorem \ref{r.gen1} for closed surfaces} Suppose $F=F_{g,0}$ with $g\ge 2$, equipped with coordinate datum $(\cP, \cD)$.
In this case $r=6g-6$. Let $\kappa: \BZ^r \embed \BZ^{r+1}$ be the group embedding  given by
$$ \kappa(n_1,\dots, n_r) = (\sum_{i=1}^{3g-3} n_i, n_1, n_2, \dots, n_r).$$
Let the order $\le$ on $\cS$ and $\BZ^r$ be the one induced from the lexicographic order of $\BZ^{r+1}$ via the embeddings
$$ \cS \overset{\nu}\longrightarrow  \cA \embed \bA \embed \BZ^r \overset{\kappa}\embed\BZ^{r+1}.$$
Since the first component is used to define the filtrations appeared in Proposition \ref{prsi},  for $\bn,\bn'\in \cAT$, there is $j$ such that
\be 
S(\bn) S(\bn')= \zeta^j S(\bn+\bn') +  G_<(\bn+\bn'), \label{eq.sum2}
\ee
where $G_<(\bn+\bn')$ is a $\BC$-span of $\{\al\in \cS \mid\al < S(\bn+\bn')\}$. This holds only for triangular $\bn, \bn'$. 
There is a better order on $\cS$.

\begin{lemma} There is a total order $\unlhd$  which makes $\BZ^r \times \BZ^r$ an ordered group,  induces an order on $\cS$ 
via $\tnu: \cS \embed \BZ^r \times \BZ^r$, and has a  
property that
 $\al \unlhd \beta$ if and only if 
$ h^k_\Omega(\al) \le h^k_\Omega(\beta)$ for sufficiently large $k$.
\end{lemma}
\def\tk{\tilde \kappa}
\begin{proof} For $\bn=(n_1,\dots, n_{6g-6})$ let $\|\bn\|_1= \sum_{i=1}^{3g-3} n_i$. For $\bn, \bm \in \BZ^r$,
$$\kappa( k \bn +\bm)=  (k\|\bn\|_1 + \|\bm\|_1, kn_1+ m_1, kn_2+ m_2,\dots, kn_r + m_r) .$$
Define the embedding $\tk:\BZ^r \times \BZ^r \embed \BZ^{2r+2}$ by
$$ \tk(\bn,\bm)= (\|\bn\|_1 ,\|\bm\|_1, n_1, m_1, n_2, m_2, \dots, n_r , m_r) .$$

The order $\unlhd$ on $\BZ^r \times \BZ^r$ induced from the lexicographic order of $\BZ^{2r+2}$ via $\tk$ satisfies the lemma.

\end{proof}

From the definition, $c_\al \al$ is the ($\unlhd$)-lead term of $x\in \KFs$ if and only if $c_\al h^k_\Omega(\al)$ is the $(\le)$-lead term of $h^k_\Omega(x)$ for 
sufficiently large $k$. This yields an ordering that gives us control of lead terms of all diagrams, not just triangular ones.

From here we have the following.
\begin{lemma} \label{r.10}
Suppose $x\in \KFs$, $\al_k\in \cS$ and $0\neq c_{k} \in \BC$ such that for sufficiently large $k$ we have
$$ h^k_\Omega(x) = c_k \al_{k} + G_{<} (\al_k).$$
Then $c_k\al_k= c_\al h^k_\Omega(\al)$ for large $k$, where $c_\al \al$ is the $(\unlhd)$-lead term of $x$.
\end{lemma}

A crucial property of the $\unlhd$ order is that its lead term is a  monoid map.

\begin{lemma} Suppose $\al, \beta\in \cS$ and $\gamma= \max_\unlhd (\supp(\al\beta))$, then $\tnu(\gamma) = \tnu(\al) + \tnu(\beta)$, i.e. $\mu(\gamma)= \mu(\al) + \mu(\beta)$ and $\eta(\gamma)= \eta(\al) + \eta(\beta)$.
\label{r.sum3}
\end{lemma}
\begin{proof}
By Theorem \ref{thm.stab} there exists $K\in\BN$ such that for all $k>K$ the diagrams
 $h^k_\Omega(\al)$ and $h^k_\Omega(\beta)$ are triangular by. By \eqref{eq.sum2},
\begin{align*}
h^k_\Omega(\al \beta) & = h^k_\Omega(\al ) h^k_\Omega(\beta)  \\
&= q^{j(k)} (S( \nu(h^k_\Omega(\al )) + \nu(h^k_\Omega(\beta )))  + G_<(\nu(h^k_\Omega(\al )) + \nu(h^k_\Omega(\beta ))).
\end{align*}
Hence by Lemma \ref{r.10} we have $h^k_\Omega(\gamma) = S( \nu(h^k_\Omega(\al )) + \nu(h^k_\Omega(\beta )))$, or
$$ \nu(h^k_\Omega(\gamma)) = \nu(h^k_\Omega(\al )) + \nu(h^k_\Omega(\beta )).$$
Using \eqref{stabdt} we have
$$ k\mu(\gamma) + \eta(\gamma) = k\mu(\al) + \eta(\al) + k\mu(\beta) + \eta(\beta).$$
It follows that $\mu(\gamma)= \mu(\al) + \mu(\beta)$ and $\eta(\gamma)= \eta(\al) + \eta(\beta)$.
\end{proof}

Define the map $\deg: \KFs \to \bA$ by 
\be \deg(x) = \eta(\al) \in \bA,  \quad \text{where} \ 
 \al= {\max}_\unlhd(\supp(x)). 
 \label{eq.defdeg2} 
\ee
We show that this is in fact a degree mapping in the sense of subsection \ref{genres}.
\begin{lemma}  \label{r.monoid1}
(a) $\deg: \KFs \to \bA$ is a monoid homomorphism.

(b) Suppose $x_1, \dots , x_d \in \KFs$ such that $\deg(x_1), \dots, \deg(x_d)$ are distinct,  then $\sum_i x_i \neq 0$.
\end{lemma}
\begin{proof} (a) follows from Lemma \ref{r.sum3} and the fact that $\bA\times \bA$, equipped with $\unlhd$, is an ordered monoid.

(b) Let $c_i \al_i$ be the $(\unlhd)$-lead term of $x_i$. Since $\deg(x_i)=\eta(\al_i)$, the $\al_i$ are distinct. It follows that $\sum_i x_i\neq 0$.
\end{proof}

\def\Sz{\cS_\zeta}
\begin{lemma} (a) Suppose $\al\in \Sz$, then $\deg(\al) \in \bAz$.

(b) Suppose $z\in \ZFs$, then $\deg(z)\in \bAz$. \label{r.l2}
\end{lemma}
\begin{proof} (a)   For sufficiently large $k$, from \eqref{stabdt} we have
\be
\eta(\al)=  \nu(h^k_\Omega(\al)) - k \mu(\al).
\label{eq.s6}
\ee
Since $\Sz$ is invariant under Dehn twists, $h^k_\Omega(\al)\in \Sz$, and 
 $\nu(h^k_\Omega(\al))\in \Az$. One the other hand, if $k\in 2m \BZ$, then $k \mu(\al)\in \Az$ since $2m \cA \subset \Az$. Hence from \eqref{eq.s6} we see that $\eta(\al)$, being the difference of two elements of $\Az$, is in $\bAz$.

(b) 
Since  $\{ T(\al) \mid \al \in \cS_\zeta\}$ is a $\BC$-basis of $\ZFs$ (see Theorem \ref{center}), we have
$$ z = \sum_{\al \in U \subset \Sz} c_\al T(\al), \quad c_\al \in \BC^*.$$
From the definition, $\max_\unlhd \supp(T(\al))= \al$. Hence $ \deg(z) = \eta(\al)$, where $\al = \max_\unlhd U$. Since $\al\in \Sz$, the result follows from (a).
\end{proof}
\begin{lemma} The monoid homomorphism $\deg: \KFs \to \bA/\bAz= \Rz$ is surjective. \label{r.l3}
\end{lemma}
\begin{proof}
Let $\overline{ \eta(\cS)}$ be the $\BZ$-span of $\eta(\cS)$. 
One has to show that $\overline{ \eta(\cS)} \supset \bA$.
From the description of $\cA$ in Section \ref{sec.Coor2} we see that
$\bA = (\bA)_1 \oplus \BZ^{3g-3}$, and $(\bA)_1$ is the set of all 
$\bn=(n_1,\dots, n_{3g-3})\in \BZ^{3g-3}$ such that whenever $P_i, P_j, P_l$ bound a pair of pants,
 $n_i+ n_j + n_l$ is even. 

The set  $\cA_1'$ of all $\bn=(n_1,\dots, n_{3g-3})\in \BN^{3g-3}$, such that whenever $P_i, P_j, P_l$ bound a pair of pants, $n_i+ n_j + n+l$ is even and $n_i \le n_j + n_l$, spans $(\bA)_1$ over $\BZ$. If $\bn\in \cA_1'$ then there is a simple diagram $\al$ lying entirely in $N(\cD)$ such that $\nu(\al) =(\bn, \vec 0)$.
By Proposition \ref{r.aden}(c), one has $\eta(\al) = \nu(\al) = (\bn, \vec0)$. It follows that
$\overline{\eta(\cS)} \supset (\bA)_1 \oplus  \{ \vec0 \}$.

Since $\eta(P_i) = - (\vec 0, \bode_i)$ by Proposition \ref{r.aden}(e), we have $\overline{\eta(\cS)} \supset \{ \vec0 \} \oplus\BZ^{3g-3}$. Thus, $\overline{\eta(\cS)} \supset (\bA)_1 \oplus \BZ^{3g-3} = \bA$.
 \end{proof}
 Theorem \ref{r.gen1} follows from Lemmas \ref{r.monoid1}, \ref{r.l2}, and \ref{r.l3}. \qed

\subsection{More on $\dz$} 

 The degree map yields a characterization of central skeins and allows the exploration of the independence of diagrams.

\begin{prop} \label{primitive} Suppose $F= F_{g,p}$ has negative Euler characteristic and a coordinate datum. Let $\zeta$ be a root of 1, with $m =\ord(\zeta^4)$. 

(a) If $\al\in \cS$ then $\deg_\zeta(\al)=0$ if and only if $\al\in \Sz$.

(b) Let $C_1, \dots, C_k$ be a sequence of disjoint non-trivial non-peripheral loops such that no two of them are isotopic. For $\bn=(n_1,\dots,n_k) \in \BN^k$ let $C^\bn= \prod_{i=1}^k (C_i)^{n_i} \in \cS$. Suppose $\dz(C^\bn)=0$.

(i) If $\ord (\zeta) \neq 0 \pmod 4$ then 
 $\bn\in m \BZ^k$. 

(ii) If $\ord (\zeta) = 0 \pmod 4$ then  $\bn\in m \BZ^k$ and $C^{\bn/m}\in \cS^\ev$.

\no{
\begin{itemize}
\item If $\ord(\zeta)\neq 0 \mod{4}$ then  $\dz(T(C^\bm))=0$  if and only if  $\bm \in m\BN^k$.
\item If $\ord(\zeta)=0 \mod{4}$ then $\dz(T(C^\bm))=0$  if and only if $\bm \in m\BN^k$, and  $C^{\bm/m}$  represents $0$ in $H_1(\overline{F};\mathbb{Z}_2)$.
  \end{itemize}
  }

\end{prop} 
\begin{proof} (a) By definition, $\dz(\al)=0$ if and only if $\deg(\al) \in \bAz$.

Case 1: $p>0$. In this case $\deg(\al) = \nu(\al)$. If $\al\in\Sz$ then $\deg(\al)=\nu(\al)\in \nu(\Sz)= \Az\subset \bAz$. Conversely, suppose $\nu(\al)\in \bAz$ then $\nu(\al) \in \bAz \cap \cA= \Az$. Hence $\al\in \Sz$.

Case 2: If $p=0$ then $\deg(\al)= \eta(\al)$.  If $\al\in \Sz$ then  $\deg(\al)= \eta(\al)\in \bAz$ by Lemma \ref{r.l2}.

Suppose $\eta(\al) \in \bAz$. For large $l$ we have
$$ \nu(h^l_\Omega(\al))= l \mu(\al) + \eta(\al).$$
When $l$ is a multiple of $2m$, one has $l \mu(\al)\in 2m \cA \subset \Az$, and the right hand side of the above is in $\bAz$. It follows that $\nu(h^l_\Omega(\al)) \in \bAz \cap \cA= \Az$. Hence $h^l_\Omega(\al)\in \Sz$. As $\Sz$ is invariant under automorphisms of $F$, we have $\al\in \Sz$.

(b) By part (a), we have $C^\bn \in \Sz$. From the definition of $\Sz$ (see Section \ref{sec.Cheb}) one has  $C^\bn= \beta \gamma^m$ where $\beta\in \Ad$ and $\gamma \in \cA$. Since there are no peripheral elements among the $C_i$ we must have $\beta=\emptyset$ and  $C^\bn= \gamma^m$. This proves $\bn \in m \BN^k$. Moreover $\gamma= C^{\bn/m}$.

If $\ord(\zeta)=0 \mod{4}$, then the definition of $\Sz$ requires $\gamma\in \Se$. Hence in this case $C^{\bn/m}\in \Se$.
\end{proof}

\def\AL{\mathfrak A}
\def\ZL{{\mathfrak Z}}
\def\tZL{{\tilde \ZL}}
 
\section{Dimension of $\KF$ over $\ZF$}\label{dimkoverz}
\subsection{Formulation of result}
Recall that for a finite type surface $F=F_{g,p}$ and a root of unity $\zeta$ with  $m =\ord(\zeta^4)$,
\be  D_\zeta(F) =  \begin{cases}
 m^{6g-6 + 2p} \qquad &\text{if }\  \ord (\zeta) \not \equiv 0 \pmod 4  \\
   2^{2g} \, m^{6g-6 + 2p} \qquad &\text{if } \ \ord (\zeta)  \equiv 0 \pmod 4.
   \label{eq.dimdef}
   \end{cases}
   \ee

\begin{theorem}   \label{thm.dim}
Suppose $F$ is a finite type surface with negative Euler characteristic and $\zeta$ is a root of 1, then
$\dim_{\ZF} \KF = D_\zeta(F)$.
\end{theorem}

\begin{remark} Let us discuss the cases excluded by Theorem \ref{thm.dim}, namely the cases when the Euler characteristic of $F_{g,p}$ is non-nagative. There are four such cases: the sphere with zero, one or two punctures and the torus. The skein algebras of the first three are commutative so they have dimension $1$ over their respective centers.  For the torus in the case where $n$ is odd this was done in \cite{AF1} and the dimension is $m^2$. The case when $m$ has residue $2$ on division by $4$ is similar and the dimension is $m^2$. Finally, when $n$ is divisible by $4$ the dimension is $4m^2$.
\end{remark}

\begin{cor} \label{r.dim=}
Let  $F=F_{g,p}$ be a finite type surface with negative Euler characteristic equipped with a coordinate datum.
Suppose $X$ is a $\tZF$-vector subspace of $\tKF$ such that $\dz(X \setminus \{0\})= \Rz$. Then $X= \tKF$.
\end{cor}
\begin{proof} Let $\cB\subset X$ be such that $\dz$ is a bijection from $\cB$ to $\Rz$. By Corollary \ref{r.dimlow}, $\cB$ is $\tZF$-linearly independent. Thus
$\dim_\tZF X \ge |\Rz|= D_\zeta(F) = \dim_\tZF \tKF$, and hence $X= \tKF$.
\end{proof}

By Corollary \ref{r.dimlow} we have $\dim_{\ZF} \KF \ge D_\zeta(F)$. To prove Theorem \ref{thm.dim} we need  to prove the converse inequality
\be 
\dim_{\ZF} \KF \le D_\zeta(F). \label{eq.dimup}
\ee
 \def\Sz{\cS_\zeta}
\subsection{Proof of Theorem \ref{thm.dim}, open surface case} Assume $p>0$. 
Fix a coordinate datum (a triangulation) $
\{e_1,\dots, e_r\}$.

Let $F_k(\KF)$ be the $\BC$-vector subspace of $\KF$ spanned by  $\{ \al\in \cS\mid \sum_{i=1}^r \nu_i(\al) \le k\}$. By Proposition \ref{r.filt}, 
$(F_k(\KF))_{k=0}^\infty$ is a filtration of $\KF$ compatible with the product.
Let $Q\subset \BR^r$ be the simplex 
$$ Q= \{ (x_1,\dots, x_r) \in \BR^r \mid x_i \ge 0,\ \sum x_i \le 1\}.$$

From Theorem \ref{center} it follows that 
\begin{align*}
&\{ T(S(\bn)) \mid \bn \in \cA \cap kQ \}  \quad \text{is a  $\BC$-basis of $F_k(\KF)$} \\
&\{ T(S(\bn)) \mid \bn \in \cA_\zeta \cap kQ \}  \quad \text{is a  $\BC$-basis of $F_k(\ZF)$}.
\end{align*}

It follows that
\begin{align}
 \dim_\BC F_k(\KF)& = |\cA \cap kQ|= |\bA \cap kQ|\\
\dim_\BC F_k(\ZF)& = |\Az \cap kQ| = |\bAz \cap kQ|.
\end{align}

Hence by  Lemma \ref{r.lim2}, there is a positive integer $u$ such that
\be 
\dim_{\ZF} \KF \le \lim_{k\to \infty} \frac{|\bA \cap kQ| }{|\bAz \cap (k-u)Q|}= |\bA/\bAz| = D(F,\zeta),
\ee
where the first identity follows from  \eqref{eq.vol3} and the second one follows from Proposition \ref{r.dim1}. This proves  Theorem \ref{thm.dim} for open surfaces. \qed

\subsection{Piecewise-rational-linear functions}
 A function $f: \BR^k \to \BR^l$ is {\em rational-linear} if there is
 is a matrix $A$ with rational entries such that  $f(x) = A x$. A function $f: \BR^k \to \BR^l$ is {\em piecewise-rational-linear} if it is continuous and there are rational-linear functions $f_1,\dots, f_k: \BR^k \to \BR$ such that on each connected component of the complement of all the hyperplanes $\{ x\in \BR^k \mid f_i(x) =0\}$, the function $f$ is equal to the restriction of a rational-linear function.
 
 For $X \subset \BR^k$ a function $h: X \to \BR^l$ is {\em piecewise-rational-linear} if there is a piecewise-rational-linear function from $\BR^k$ to $\BR^l$ restricting to $h$.

It is clear that sums of piecewise-rational-linear functions are piecewise-rational-linear, and that a piecewise-rational-linear function $h$ is positively homogeneous, i.e. $h(t x) = t h(x)$ for all real $ t\ge 0$. 

A {\em rational convex polyhedron} is the convex hull of a finite number of points in $\BQ^n$.
The following properties are easy consequences of the definition.

\begin{prop} \label{r.rpwl}
Suppose $h_1, \dots, h_l : \BR^n \to \BR$ are piecewise-rational-linear, $c_1,\dots, c_l\in \BQ$ and
$$ Q =\{ x\in \BR^n \mid h_i(x) \ge c_i\}.$$

(a) If $Q$ is not bounded, then $Q$ contains  a set of the form
$\{ tx \mid t \in \BR_{\ge 0}\}$
for some non-zero $x\in \BQ^n$. We call such set a {\em rational ray}. 

(b) 
If $Q$ is bounded, then $Q$ is the union of a finite number of rational convex polyhedra. \end{prop}

\subsection{Infinite sector $Q_\infty$} Suppose $F=F_{g,0}$ with $g\ge 2$. Fix a coordinate datum $(\cP,\cD)$ which gives  the DT coordinate map
$$ \nu : \cS \embed \BZ^{6g-6}.$$
Here $\cP=(P_1,\dots, P_{3g-3})$ is a pants decomposition.
The set $\cA= \nu(\cS)$
consists of all points
$(x_1, \dots, x_{6g-6}) \in \BZ^{6g-6}$ satisfying 

\begin{itemize}
\item [(i)] $x_i \ge 0$ for $i=1,\dots, 3g-3$,
\item [(ii)] if $x_i = 0$ for some $i=1,\dots, 3g-3$, then $x_{i+3g-3}\ge0$,

\item[(iii)]
if $P_i, P_j, P_l$ bound a pair of pants, then $x_i + x_j + x_l$  is even.
\end{itemize}

Let $Q_\infty$ be the set of all $(x_1,\dots, x_{6g-6})\in \BR^{6g-6}$ satisfying conditions (i) and   (ii) above. Note that we allow points in $Q_\infty$ to have  real coordinates. 

 For any submonoid $X$ of $\BZ^{6g-6}$, let $\overline X$ be the subgroup generated by $X$. The set $Q_\infty$ was introduced so that 
$
\cA = \bA \cap Q_\infty.   $

\begin{lemma} (a) For any subset $Q'\subset Q_\infty$ one has
\be 
\cA \cap Q' = \bA \cap Q' 
\ee
(b) If $x \in Q_\infty\cap \BZ^{6g-6}$ then $ 2x \in \cA$.
\end{lemma}
\proof (a) 
 $\ \bA \cap Q' = \bA \cap (Q_\infty \cap Q') = (\bA \cap Q_\infty) \cap Q'= \cA \cap Q'$.
 
 (b) follows from (iii) above.\qed

\subsection{DT coordinates and geometric intersection numbers} 
\begin{lemma}  \label{r.plDT} 
(a) There are $6g-6$ loops $P_{3g-2},\dots, P_{9g-9}$ such that any $\al\in \cS$ is totally determined by the collection $(I(\al, P_i))_{i=1}^{9g-9}$. 

(b) Moreover, for each $i$ the function $\cA\to \BR$,
defined by $ \bn \to I(S(\bn), P_i)$, is piecewise-rational-linear.
\end{lemma}
\begin{proof}
 (a) 
 follows, for example, from the computations in Section 4 of \cite{LS}.
 
 (b) The more general fact: ``For any simple closed curve $\al$, the function $\cA\to \BR$,
defined by $ \bn \to I(S(\bn), \al)$, is piecewise-rational-linear" is well-known. It was formulated as Theorem 3 in \cite{Thurston} without proof. Here is a short proof based on \cite{Penner}. First if $\al$ is one of $P_i$ with $i \le 3g-3$ then the statement is obvious as $I(S(\bn), P_i)=n_i$. Suppose now $\al$ is an arbitrary simple closed curve. Choose a coordinate datum $(\cP', \cD')$ such that $\al$ is a curve in $\cP'$. By \cite{Penner}, the change from DT coordinates associated with $(\cP', \cD')$ to the one associated with $(\cP,\cD)$ is piecewise-rational-linear. The result follows.

\end{proof} 

It follows that  there is a piecewise-rational-linear $h: \BR^{6g-6} \to \BR$ such that 
\be 
h(\bn) = \sum_{i=1}^{9g-9} I(S(\bn), P_i) \quad \text{for all $\bn \in \cA$}.
\ee

Let $Q:= \{ x \in Q_\infty \mid h(x) \le 1\} $. Since $(I(P_i, \al))_{i=1}^{9g-9}$ totally determine $\al\in \cS$,  the set $ \cA \cap kQ$ is finite for any $k \ge 0$.

\begin{lemma}\label{r.Q}  The set $Q$ is the union of a finite number of convex polyhedra. Moreover, $Q$ has positive volume in $\BR^{6g-6}$.
\end{lemma}
\begin{proof} Let us prove that $Q$ is bounded. Suppose to the contrary that $Q$ is not bounded. By Proposition \ref{r.rpwl}(a), $Q$ contains a rational ray, which in turns contains infinitely many points whose coordinates are even integers. Since each such point is in $\cA$, the set $\cA \cap Q$ is infinite, a contradiction.  Thus $Q$ is bounded, and by Proposition \ref{r.rpwl}(b), $Q$ is the union of a finite number of convex polyhedra.

Choose  $\al\in \cS$ with $\nu(\al)= (n_1,\dots,n_{6g-6})$ satisfying $n_i >0$ and whenever $P_i, P_j, P_l$ bound a pair of pants then $n_i < n_j + n_l$.  Let $h(\nu(\al))= k$ then the point $\nu(\al)/(k+1)$ is an interior point of $Q$. Hence $Q$ has positive volume. 
\end{proof}

\subsection{Proof of Theorem \ref{thm.dim}, closed surface case} Let  $F_k(\KF)$ be  the $\BC$-subspace spanned by 
$\{ \al\in \cS \mid h(\nu(\al)) \le k\}$. 
By Proposition \ref{r.filt}, $(F_k(\KF))_{k=0}^\infty$ is a filtration of $\KF$ compatible with the product. 
Then $\{ S(\bn) \mid \bn \in \cA \cap k Q\}$ is a $\BC$-basis of $F_k(\KF)$, hence
\be
\dim_\BC F_k(\KF) = | \cA \cap kQ| = |\bA \cap kQ|. \label{eq.ine1}
\ee

If $\bn\in \Az \cap kQ  $, then by Theorem \ref{center} one has $T(S(\bn)) \in \ZF \cap F_k(\KF)  = F_k(\KF) $. Since the collection $\{ T(S(\bn)), \bn \in \Az \cap kQ \}$ is $\BC$-linearly independent, we have
\be 
\dim_\BC F_k (\ZF) \ge |\Az \cap kQ| = |\bAz \cap kQ|. \label{eq.ine2}
\ee  
Using Lemma \ref{r.lim2} then \eqref{eq.ine1} and \eqref{eq.ine2}, we get, for some integer $u>0$,
$$ \dim_{\ZF}{\KF} \le \lim_{k\to \infty } \frac{\dim_\BC F_k(\KF)}{\dim_\BC F_{k-u} (\ZF)} \le \lim_{k\to \infty } \frac{|\bA \cap kQ|}{|\bAz \cap (k-u)Q|} .$$
The latter, by \eqref{eq.vol3}, is  $[\bA: \bAz]$, which is equal to $D_\zeta(F)$ by Proposition~\ref{r.dim1a}. Thus, $ \dim_{\ZF}{\KF} \le D_\zeta(F)$, completing the proof of Theorem~\ref{thm.dim}.\qed

\section{Commutative subalgebras of $\tKF$}   \label{sec.42}  

In this section we study commutative subalgebras generated by collections of disjoint loops and describe their bases.

For a finite type surface $F$ and a root of unity $\zeta$ recall that $\tZF$ is the field of fractions of the center $\ZF$ of $\KF$, and $\tKF= \KF \ot _\ZF \tKF$ is a division algebra. Recall that if $F$ is a surface with punctures $\{p_1,\dots, p_k\}$ then $\overline{F}=F\cup\{p_1,\dots, p_k\}$.
 
\begin{prop}    \label{r.com}  
Suppose $C_1, \dots C_k$ are non-peripheral, non-trivial, disjoint, pairwise non-isotopic loops on a finite type surface $F=F_{g,p}$ of negative Euler characteristic. Let $\zeta$ be a root of unity with $m=\ord(\zeta^4)$. Let $\cC$ be the $\tZF$-subalgebra of $\tKF$ generated by $C_1, \dots , C_k$.
For $\bn=(n_1,\dots,n_k)\in \BN^k$ let $C^\bn= \prod_{i=1}^k (C_i)^{n_i} \in \cS$.

(a) Suppose  $\ord(\zeta) \neq 0 \pmod 4$ then $\dim_\tZF \cC = m^k$ and the set 
\begin{equation} B=\{ C^\bn  \mid   0\le n_i < m\}\end{equation}
 is a basis of $\cC$ over $\tZF$.

(b) Suppose $n=0 \pmod 4$ then $\dim_\tZF \cC = 2^t m^k$, where $t$ is the $\BZ_2$-rank of the subgroup 
$H$ of $H_1(\bF,\BZ_2)$ generated by $C_1, \dots, C_k$.

 Assume that after a re-indexing $\{C_1,\dots,C_t\}$ is a basis for $H$. The set
\begin{equation}\label{bml} B =\{ C^\bn  \mid  n_i < 2 m \ \text{for} \ i \le t, \  {n_i < m}\  \text{for} \ i > t\}\end{equation}
is a basis of $\cC$ over $\tZF$.
\end{prop}

\begin{proof} 

(a) In this case $(C_i)^m \in \Sz$ for all each $i$. By Theorem \ref{center}, 
 $T_m(C_i)=T((C_i)^m) \in \ZF$, which implies that the degree of $C_i$ over $\tZF$ is $\le m$. Hence  $B$  spans $\cC$  as a vector space over $\tZF$. 
 
 By Corollary \ref{r.dimlow}, to prove that $B$ is linearly independent it is enough to show that $\dz(x), x \in B$,  are distinct. Assume $\dz(C^\bn) = \dz(C^{\bn'})$. Let ${\bf m}$ be the $k$-tuple all of whose entries are $m$. Since $\dz$ is a monoid homomorphism and $\dz(C^\bm)=0$, we have $\dz(C^{\bm-\bn + \bn'})=0$. By Proposition \ref{primitive}(b), for each $i$ we have
$$m-n_i+ n'_i =0\pmod m.$$  Since $0\leq n_i,n_i'\leq m-1$, the only way this can happen is if $n_i=n_i'$.

(b) Let $\cC_0$ be the $\tZF$-subalgebra generated by $C_1, \dots, C_t$. Since 
 $T_{2m}(C_i)\in \ZF$, the set $B_0=\{ C_1^{n_1} \dots C_t^{n_t} \mid n_i < 2m\}$ spans $\cC_0$ over $\tZF$.  Suppose $i>t$. There are $j_1, \dots j_l \le t$ such that the simple diagram $\beta= C_i \cup  C_{j_1} \cup \dots \cup  C_{j_l}$ is even. This implies that $\beta^m\in \Sz$. Hence  $T(\beta^m) \in \ZF$ by Theorem \ref{center}. Using the definition of $T(\beta^m)$, 
$$\ZF \ni T(\beta^m) = T_m(C_i) \left[T_m(C_{j_1}) \dots T_m(C_{j_l}) \right].$$
The element in the square bracket is in $\cC_0$. It follows that $T_m(C_i)\in \cC_0$, which implies that the degree of $C_i$ over $\cC_0$ is less than equal to $m$ for each $i\ge t+1$. Hence  $B_1:=\{ C_{t+1}^{n_{t+1}} \dots C_k^{n_k} \mid n_i < m\}$ spans $\cC$ over $\cC_0$. Combining the spanning sets $B_0$ and $B_1$, we get that $B$ spans $\cC$ over $\tZF$. 

Let us show that $\dz(x), x \in B$, are distinct. Suppose $\dz(C^\bn )= \dz(C^{\bn'})$. Let $\bm=(m_1,\dots,m_k)$ where $m_i=2m$ for $i \le t$ and $m_i=m$ for $i>t$. Then $\dz(C^\bm)=0$. It follows that $\dz(C^{\bm -\bn + \bn'})=0$. By Proposition \ref{primitive}(b), we have $\bm -\bn + \bn'\in m\BZ^k$. This forces $n_i = n'_i$ for $i>t$ as in this case $ m_i -n_i + n'_i$ is sandwiched between $1$ and $2m-1$. Further  $C^{(\bm -\bn + \bn)/m}$ is even by Proposition \ref{primitive}(b). Since $C_1, \dots, C_t$ are linearly independent over $\BZ_2$ in $H_1(\bF,\BZ_2)$, for each $i\le t$,  $(m_i -n_i + n'_i)/m$ is even. As $m_i=2m$ and $0\le n_i, n'_i <2m$, this forces $n_i= n'_i$. Thus $\dz(x), x \in B$, are distinct, and by Corollary \ref{r.dimlow} the set $B'$ is linearly independent over $\tZF$.
 \end{proof}

\begin{cor}  Assume the conditions of Proposition \ref{r.com}.
 \label{r.deg}
 Let $\cC'$ be the $\tZF$-subalgebra of $\tKF$   generated by $C_1, \dots, C_{k-1}$.  The minimal polynomial of $C_k$ over $\cC^{'}$ is of the form
$T_{m'}(x) - u$, where $u\in \cC'$ and 
$$ m' = \begin{cases}
 m \quad & \text{if } \ord(\zeta)  \neq 0 \pmod 4, \ \text{or}\  n = 0 \pmod 4 \ \text{and} \ k>t, \ \\
 2m   & \text{if} \  \ord(\zeta) =  0 \pmod 4 \ \text{and} \ k=  t. \end{cases}   
 $$
 Moreover $u$ is transcendental over $\BQ$.
  \end{cor}
 
  \begin{proof} Since $\cC=\cC'(C_k)$, the degree of $C_k$ over $\cC'$ is 
  $$[\cC:\cC']= \frac{\dim_\tZF \cC}{\dim_\tZF \cC'} \ , $$ which is equal to $m'$ using the  formula for $\dim_\tZF C$ and $\dim_\tZF C'$ given by  Proposition \ref{r.com}. In the proof of Proposition \ref{r.com} we see that $T_{m'}(C_k) = u\in \cC'$. Hence  $T_{m'}(x) -u$ is the minimal polynomial of $C_k$ over $\cC'$.

  Suppose $u= T_{m'}(C_k)$ is algebraic over $\BQ$. Since $m'>0$ this implies $C_k$ is algebraic over $\BQ$.  But  
  $\{C_k^i, i\geq 0\}$ is a subset of $\cS$, which is a $\BC$-basis of $\KF$ and  hence the non-trivial $\BQ$-linear combination of these elements is  never $0$. This shows $u$ is transcendental over $\BQ$.
\end{proof}

  \section{Calculation of the reduced trace}\label{calcoftr}
   Let $F$ be a finite type surface and $\zeta$ a root of unity. Since  $K_{\zeta}(F)$ is finitely generated as a module over its center $\ZF$,  it has a reduced trace.   The goal of this section is to find a formula for computing it.


By Theorem \ref{center}  the set $\{ T(\al) \mid \al \in \cS\}$ is a $\BC$-basis of $\KF$. Therefore it is enough to calculate $\tr(T(\al))$ for each $\al \in \cS$.

\begin{theorem} 
\label{thm.trace} 
Let $F$ be a finite type surface, $\zeta$ be a root of 1,  and 
$$
\tr: \tilde{K}_{\zeta}(F)\rightarrow \tZ_{\zeta}F)
$$
be the reduced trace. For $\al\in \cS$ one has
\be
\tr(T(\al))= \begin{cases}
T(\al) \quad &\text{if $T(\al)$ is central, i.e., }  \al\in \Sz \\
0     & \text{otherwise.}
\end{cases}
\ee
 \end{theorem}
 First consider the case when $F_{g,p}$ has non-negative Euler characteristic.
The skein algebras of  $F_{g,p}$  for $g=0$ and $p=0,1,2$ are commutative, so the result is trivial. For $F_{1,0}$ and $n$ not divisible by $4$ this is proved in \cite{AF1}. The remaining case of $K_{\zeta}(F_{1,0})$ and $n$ divisible by $4$ can be proved using similar methods. Hence we will assume that $F$ has positive Euler characteristic.
 \no{There are only four connected oriented surfaces of nonnegative Euler characteristic. Of these only the closed surface of genus one has noncommutative skein algebras, in the other cases as the skein algebra is a one dimensional vector space over its center, the normalized trace is the identity map. In the case of the torus, when the order of $\zeta$ has remainder $2$ on division by $4$, the normalized trace was computed in \cite{AF1} and the answer is the same as in Theorem \ref{thm.trace}.  The case when $n$ works similarly to this case.  Hence we focus on the case when $n$ is divisible by $4$. Let $(a,b)_c\in K_\zeta(\Sigma_{1,0})$ where $(a,b)\in \mathbb{Z}\times \mathbb{Z}$ be the noncommutative cosine from \cite{FG}.  We have $(a,b)_c=(-a,-b)_c$.  Also,
 \be (a,b)_c*(e,f)_c=\zeta^{\left|\begin{matrix} a & b \\ e & f \end{matrix}\right|}(a+e,b+f)_c+ \zeta^{-\left|\begin{matrix} a & b \\ e & f \end{matrix}\right|}(a-e,b-f)_c.\ee
 Letting $m=n/4$, it is easy to see that $(a,b)_c$ is central in $K_{\zeta}(\Sigma_{1,0})$ if and only if $2m$ divides both $a$ and $b$.  The trace of any $(a,b)_c$ where $a$ or $b$ is not divisible by $m$ is $0$ as its shifts the grading by $\mathbb{Z}_2$-homology. Hence we only need to understand the case or $(a,b)_c$
 where $a$ and $b$ are divisible by $m$.  This is a four dimensional vector space, over its center and the computation of trace is immediate and agrees with the formula in the theorem above.}
 
\subsection{Lemma on traces}  \label{sec.41} Recall that $T_l(x)$ is defined in Section \ref{sec.Cheb}.
\begin{lemma}
\label{r.trace0}
Suppose $k_1\subset k_2$ are finite field extensions of a field $k_0$ and $x_1\in k_1, x_2 \in k_2$.\\
(a) If\  $\Tr_{k_2/k_1}(x_2)=0$ then\  $\Tr_{k_2/k_0} (x_1 x_2)=0$.\\
(b) Assume the minimal polynomial of $x_2$ over $k_1$ is $T_l(x) - u$, where $u\in k_1$  is transcendental over $\BQ$,  and $l\ge 2$. For $0< s < l$ we have
\be
\Tr_{k_2/k_0} (x_1 T_s(x_2))=0.
 \ee
\end{lemma}
\begin{proof}
(a) A  property of the trace is that for any $x\in k_2$ we have
$$\Tr_{k_2/k_0}(x)= \Tr_{k_1/k_0} \left( \Tr_{k_2/k_1}(x)\right),$$
 see eg \cite{Mc}. With $x= x_1 x_2$, we have
$$   \Tr_{k_2/k_0}(x_1 x_2)= \Tr_{k_1/k_0} \left( \Tr_{k_2/k_1}(x_1 x_2)\right) = \Tr_{k_1/k_0} \left(x_1  \Tr_{k_2/k_1}(x_2)\right) =0.$$
(b) From (a) it is enough to show that $\Tr_{k_2/k_1}(T_s(x_2))=0$.

Let $t$ be the smallest positive integer such that $l|ts$.
Note that $ t\ge 2$. Denote $m= ts/l$. Define $u_0=1$ and $u_i= T_i(u)$ for $i\ge 1$.
 
{\em Claim.} The minimal polynomial of $y:= T_s(x_2)$ over $k_1$ is $P=T_t(x) - u_m$.

Assume the claim for now.  Since $t\ge 2$ and $T_t$ is either even or odd polynomial, the second-highest coefficient of $T_t-v$ is 0. By Proposition \ref{r.trace1}(a), we have $\Tr_{k_2/k_1}(T_s(x_2))=0$. Thus (b) follows from the claim.

{\em Proof of the Claim.} First note the $P(y)=0$. In fact, we have
$$ T_t(y)= T_t(T_s(x_2))= T_{ts}(x_2)= T_{ml}(x_2)= T_m(T_l(x_2))= T_m(u)=u_m,$$
which shows $P(y)=0$.
Let us show that no polynomial $Q(x) $ of degree $< t$ can annihilate $y$. Since $\{T_i(x)\}$ forms a basis, we can write $Q(x)= \sum_{i=0}^ {d} c_i T_i(x)$ with $d<t$, $c_i\in k_1$, and $c_d=1$.
We have
\be
0= Q(y) =  \sum_{i=0}^ {d} c_i T_{is}(x_2).
\ee
Since $T_l(x) -u$ is the minimal polynomial of $x_2$, we have
\begin{equation}   \sum_{i=0}^ {d} c_i T_{is}(x) \equiv 0 \mod {(T_l(x) -u)}.
\end{equation}

Let
$R_i(x)$ be the remainder obtained upon dividing $T_{is}$ by $T_l(x) -u$, then we must have
\begin{equation}  \label{eq.04a}
 \sum_{i=0}^ {d} c_i  R_i(x) = 0. \end{equation}
To finish the proof we need another lemma: 
 \begin{lemma} \label{r.rrr4}  Suppose $ 0 \le r <l$. When  $T_{ql+r}(x)$ is divided by $(T_l(x) -u)$, the remainder is    $S_q(u) T_r (x) - S_{q-1}(u) T_{l-r}(x) $, 
  where $S_i(x)$ is the Chebyshev polynomial of the second kind defined recursively by
  $$ S_0=1, S_1(x)=x, S_n(x) = x S_{n-1}(x) - S_{n-2}(x).$$
  \begin{proof}[Proof of Lemma \ref{r.rrr4}]  One can easily check that
\be
S_i (u) = \sum _{0 \le j \le (i/2)} u_{i- 2j}=u_i + u_{i-2 } + \dots. \ee
Using $T_m T_n = T_{m+n} + T_{|m-n|}$, we get
\begin{align*}
T_{ql+r}(x)  &=  T_r(x) T_{ql}(x) - T_{ql-r}(x) =  T_r(x) T_{q}(T_l(x)) - T_{(q-1)l+l-r}(x) \\
&\equiv  T_r(x) u_q  - T_{(q-1)l+l-r}(x)   \pmod {(T_l(x) -u)}\\
&\equiv  T_r(x) u_q  - T_{l-r}(x) u_{q-1} + T_{(q-2)l+r}(x)   \pmod {(T_l(x) -u)} \\
&\equiv  T_r(x) u_q  - T_{l-r}(x) u_{q-1} +  T_r(x) u_{q-2} - T_{(q-3)l+l-r}(x)   \pmod {(T_l(x) -u)} \\
&\equiv T_r(x) (u_q + u_{q-2} + u_{q-4} + \dots) - T_{l-r}(x)  (u_{q-1} + u_{q-3} + \dots) \pmod {(T_l(x) -u)} \\
&\equiv  S_q(u) T_r (x) - S_{q-1}(u) T_{l-r}(x) \pmod{(T_l(x) -u)}.
 \end{align*}  \end{proof}

 \no{
\be
T_{ql+r}(x)  \equiv S_q(u) T_r (x) - S_{q-1}(u) T_{l-r}(x) \pmod{(T_l(x) -u)}.
\ee
}
 \end{lemma}

 Suppose $ds = ql + r$, with $0\le r <l$.  By Lemma \ref{r.rrr4},
  \be 
 R_d = S_q(u)  T_r(x)  - S_{q-1}(u) T_{l-r}(x). 
 \ee
 Note that 
 
 (*) there is no  index $j \in [0,d-1]$ such that $js$ has  remainder $r$ when divided by $l$.
 
 Consider two cases: (i) $r=l-r$ and (ii) $r\neq l-r$.
 
 (i) $r=l-r$.  Then $R_d= (S_q(u)   - S_{q-1}(u))   T_r(x)$. From (*) we see that no index $j \neq d$ contributes to the term $T_r(x)$ in \eqref{eq.04a}. Hence      $S_q(u)   - S_{q-1}(u) =0$, contradicting the fact that $u$ is transcendental over $\BQ$.

  (ii) $r\neq l-r$. There is exactly one index $j\in [0,d-1]$ such that $js$ has remainder $l-r$  when  divided by $l$, which is $j= (t-d)$. Suppose  $ (t-d) s = q' l + (l-r),$
  then 
  $$ R_{t-d} = S_{q'}(u)  T_{l-r}(x)  - S_{q'-1}(u) T_{r}(x).$$
By looking at the coefficients of $ T_r(x) $ and $ T_{l-r}(x) $ in \eqref{eq.04a}, we get
(with    $c= c_{t-d}$)
 \begin{align}
 S_q(u)  - c S_{q'-1}(u) &= 0 \\
 -S_{q-1}(u) + c S_{q'}(u) &= 0. 
 \end{align}    
 Multiply the first by $S_{q'}(u)$, the second by $S_{q'-1}(u)$, and sum up the two, we get
 $$   S_q(u) S_{q'}(u) -  S_{q'-1}(u) S_{q-1}(u)=0
 $$
 contradicting the fact that $u$ is transcendental over $\BQ$. \end{proof}


\subsection{Proof of Theorem \ref{thm.trace}}

\begin{proof} Note that if $z\in \ZF$ then $\tr(z)= z$, and more generally
\be
 \tr(zx)= z\tr(x) \quad\text{ if } \ z\in \ZF. \label{eq.cen1}
 \ee
 Hence we  assume that $T(\al)\not \in \ZF$ and we will show $\tr(T(\al) )=0$.   
 
 Assume $\al = \prod_{i=1}^k (C_i)^{m_i}$, where $C_1, \dots, C_k$ are non-trivial loops, no two of which are isotopic, then $T(\al) = \prod_{i=1}^k T_{m_i}(C_i)$. If a component $C_i$ is peripheral then $C_i\in \ZF$, and \eqref{eq.cen1} shows that one reduces to the case when non of $C_i$ is peripheral.
 
 From the product to sum formula  \eqref{eq.Che1}, we have
 \begin{align}
 T_n (x) &= T_{2m}(x) T_{n-2m} - T_{|(n -2m)-2m|}(x) \quad \text{if} \ n \ge 2m \label{eq.red2}\\
 T_n (x) &= T_{m}(x) T_{n-m} - T_{|(n -m) -m|}(x) \quad \text{if} \ n \ge m. \label{eq.red1}
 \end{align}
 
  Let $\cC$ be the $\tZF$-subalgebra of $\tKF$ generated by $C_1,\dots, C_k$ and $\cC'$ be the subalgebra generated by $C_1,\dots, C_{k-1}$. Consider several cases. 
 
 (a) Suppose $n\neq 0 \pmod 4$. In this case $T_m(C)\in \ZF$ for any loop $C$ on $F$. Using  \eqref{eq.red1} and \eqref{eq.cen1}  we reduce to the case  $m_i < m$ for all $i$.
 By Corollary \ref{r.deg} the minimal polynomial of $C_k$ over $\cC'$ is $T_m -u$ for some $u\in \cC'$ and $u$ is transcendental over $\BQ$. Lemma \ref{r.trace0} with $k_0= \tZF, k_1=\cC', k_2= \cC, x_1= \prod_{i=1}^{k-1} T_{m_i}(C_i)$, and $x_2=C_k$ shows that $\tr(T(\al) )=0$.
 
 (b) Suppose $n= 0 \pmod 4$. 
 Since $T_{2m}(C_i) \in \ZF$,   using  \eqref{eq.red2} and \eqref{eq.cen1}  we reduce to the case $m_i < 2m$ for all $i$.

 (i)  Suppose $C_1, \dots, C_k$ are $\BZ_2$-linearly independent in $H_1(\bF,\BZ_2)$.   By Corollary \ref{r.deg} the minimal polynomial of $C_k$ over $\cC'$ is $T_{2m} -u$ for some $u\in \cC'$  transcendental over $\BQ$. Lemma \ref{r.trace0}
 with $k_0= \tZF, k_1=\cC', k_2= \cC, x_1= \prod_{i=1}^{k-1} T_{m_i}(C_i)$, and $x_2=C_k$ 
 shows that $\tr(T(\al) )=0$.

 (ii) Suppose $C_1, \dots, C_k$ are not $\BZ_2$-linearly independent in $H_1(\bF,\BZ_2)$.  There are indices $j_1,\dots, j_l$  such that $\sum_{i} C_{j_i} =0$ in $H_2(\bF,\BZ_2)$. This implies that $ \prod_{i} T_m(C_{j_i})\in \ZF$. If all $m_{j_i} \ge m$, then $2m > m_{j_i} \ge m$. Taking the product of $l$  Equations  \eqref{eq.red1} with $n= m_{ji}$ then using \eqref{eq.cen1},  we  reduce to the case when there is $i$ such that $m_{j_i} <m$. 
 
 Re-indexing, we assume that $j_i=k$. Thus $m_k < m$, and $C_k$, as an element of $H_1(\bF,\BZ_2)$, is in the $\BZ_2$-span  of $C_1, \dots, C_{k-1}$. 
 By Corollary \ref{r.deg} the minimal polynomial of $C_k$ over $\cC'$ is $T_{m} -u$ for some $u\in \cC'$ transcendental over $\BQ$. Again Lemma \ref{r.trace0}
 shows that $\tr(T(\al) )=0$.
\end{proof}

 \subsection{Trace  and $\dz$}

  The following is a consequence of Theorem  \ref{thm.trace}.

 \begin{prop}\label{zerores} If $x \in \KF^*$   and $\dz(x) =0$, 
    then $\tr(x)\neq 0$. \end{prop}  
    \begin{proof}
 This is proved for punctured surfaces and $n$ not divisible by $4$ in \cite{FK}, Lemma 3.8. Other cases are similar, with a stable lead term replacing the lead term in the argument.
 
    \end{proof}
  

 The following theorem extends the exhaustion criterion from \cite{FK}.

\begin{theorem}\label{exhaustion} Suppose $\mathcal{B}$ be a collection of nonzero elements in   a $\tZF$-subalgebra $A$ of $\tKF$.
\begin{itemize}

\item[(i)] If $\dz(\mathcal{B})= \dz(A^*)$ then $\mathcal{B}$  spans ${\tKF}$ over $\tZF$.

\item[(ii)] The dimension of $A$  over $\tZF$ is equal to  $|\dz(A)|$.

\end{itemize}

\end{theorem}
\no{\blue{We don't need this theorem in the sequence. True, but it is nice on its own (:}}

\proof (i) 
Since $\dz(A^*)$ is a group, for every $x\in A^*$ there exists $b \in \mathcal{B}$ such that $\dz(xb) =\dz(x) + \dz(b)=0$. By Proposition \ref{zerores}, $\tr(xb)\neq 0$. Therefore the set $\mathcal{B}$ exhausts the bilinear form given by the trace on $A\otimes_\tZF A$. Consequently $\mathcal{B}$ spans $A$.

(ii) Choose a subset $\mathcal{B}$  of $A^*$ such that $\dz|_{\mathcal{B}}:\mathcal{B}\rightarrow \dz(A^*)$ is bijective. By (i), $\mathcal{B}$ spans $A$. By Proposition \ref{r.dimlow}(a), $\mathcal{B}$ is linearly independent over $\tZF$. Hence $\cB$ is a basis. \qed

\section{Pants subalgebra decomposition}\label{pantedec}

In this section we  give a splitting of the localized skein algebra as a module over its center. 
Throughout, $F=F_{g,p}$ is a finite type surface with negative Euler characteristic, with or without punctures,  $\zeta$ is a root of unity, and $m= \ord(\zeta^4)$.

\subsection{The Splitting}
Recall that a pants decomposition of $F$ is a collection of curves $\cP=\{ C_1,\dots, C_{3g-3+p}\}$  such that each component of the complement of $\cP$ in $F$ is a planar surfaces of Euler characteristic $-1$.

Given a pants decomposition $\cP$ of $F$, let $\Cz(\cP)$ be the  $\tZF$-subalgebra of $\tKF$  generated by  the curves in $\cP$. By Proposition \ref{r.com} and Theorem \ref{thm.dim}
$$ \dim_{\tZF} \Cz(\cP)=  \sqrt{\dim_\tZF \tKF}.
\no{=\begin{cases}
 m^{3g-3+p} \quad &\text{if } \ \ord(\zeta) \neq 0 \pmod 4\\
 2^g  m^{3g-3+p} \quad &\text{if } \ \ord(\zeta)= 0 \pmod 4,
\end{cases}}$$
 Hence $\Cz(\cP)$ is a maximal commutative subalgebra of the division algebra $\tKF$.
In \cite{FK} the first two authors constructed a splitting of $\tKF$, when $F$ has at least one puncture and $\ord(\zeta) \neq 0 \pmod 4$.  
 Here we prove  that this decomposition works for all surfaces and all roots of unity.

\begin{theorem}  \label{r.decomp}
Let $F$ be a finite type surface of negative Euler characteristic. 
There exist two pants decompositions $\cP$ and $\cQ$ of $F$ such that for any root of unity $\zeta$  the $\tZF$-linear map
 \be 
  \psi:\Cz(\cP) \otimes _{\tZF} \Cz(\cQ) \to \tKF, \quad \psi(x\otimes y) \to xy, 
  \label{eq.psi}
  \ee
 is a $\tZF$-linear isomorphism of vector spaces.
  \end{theorem}

\def\bND{ N'(\cD)}
\def\oND{\mathring N(\cD)}
\def\Cd{\BC[\partial]}
\def\CP{\BC[\cP]}
\def\CQ{\BC[\cQ]}
\subsection{Pants decompositions} Let $\Cd$ be the $\BC$-subalgebra of $\KF$ generated by peripheral loops. For a pants decomposition $\cP$ let $\BC[\cP]$ be the $\BC$-subalgebra of $\KF$ generated by loops in $\cP$. For a set $U\subset \BZ^r$ let $\bU$ be the $\BZ$-span of $U$.

\begin{lemma} \label{r.p1} Suppose $F_{g,p}$ is a finite type surface with negative Euler characteristic. There exist a coordinate datum for $F_{g,p}$ and two pants
decompositions $\cP$ and $\cQ$ such that 
\be \bA = \overline{\deg(\Cd^*)} + \overline{ \deg(\BC[\cP]^*)} + \overline{\deg (\BC[\cQ]^*}).
\label{eq.full6}
\ee
It is understood that if $p=0$ then $\Cd$ is the 0 vector space.

\end{lemma}
 
\begin{proof} 
To prove the lemma we use the following result that follows from the computation of the determinants in the proof of Theorem 4.5 in \cite[Section 4]{FK}:
\begin{lemma}\label{r.full1}
Suppose $F_{g,p}$ has negative Euler characteristic and $p>0$. There exist a coordinate datum $(a_i)_{i=1}^{6g-6+3p}$ and two pants decompositions $\cP, \cQ$ such that the $\BZ$-span of $\nu(\cP), \nu(\cQ)$ and $\Ad$ 
has index $2^{4g-5+2p}$ in $\BZ^{6g-6+3p}$. In other words,  if $C_1, \dots, C_{6g-6+3p}$ it the set of curves consisting of  components of $\cP, \cQ$ and the $p$ peripheral loops, then
\be 
\det\left[  I(a_i, C_j)_{i,j=1}^{6g-6+3p}   \right] = 2^{4g-5+2p}.
\label{eq.index}\qed
\ee
\end{lemma}
\no{\blue{We need to add a proof for the case when $p>1$. And do not use color in Figure as it is very costly}.}

Pants decompositions for a  surface of genus $3$ and $1$ puncture are shown in Figure \ref{PandQ}. Additional curves needed for more than one puncture are shown in Figure \ref{punct} for the case of four punctures. For a detailed description of those curves see \cite{FK}.
\begin{figure}
\begin{center}\scalebox{0.5}{\includegraphics{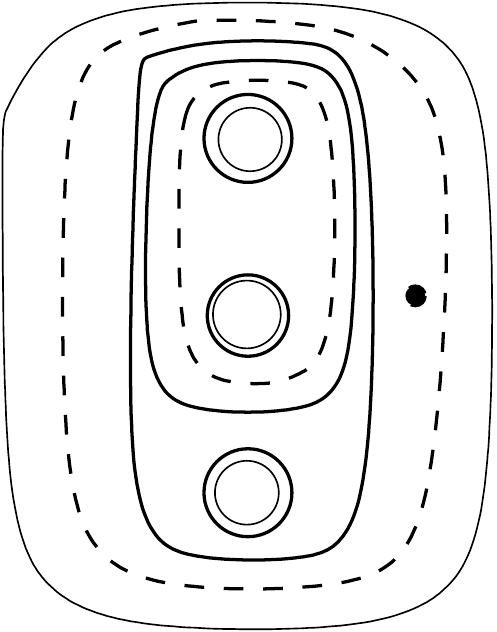}}\hspace{0.1in}\scalebox{0.5}{\includegraphics{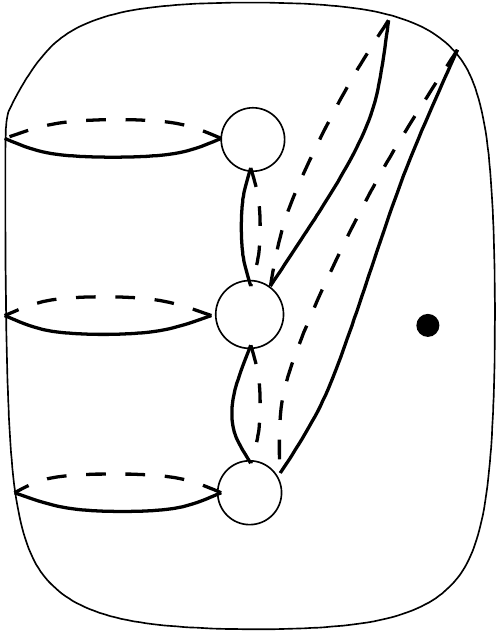}}\end{center}
\caption{$P$  and $Q$ for a surface of genus $3$ and $1$ puncture}\label{PandQ}
\end{figure}

\begin{figure}
\begin{center}\scalebox{0.5}{\includegraphics{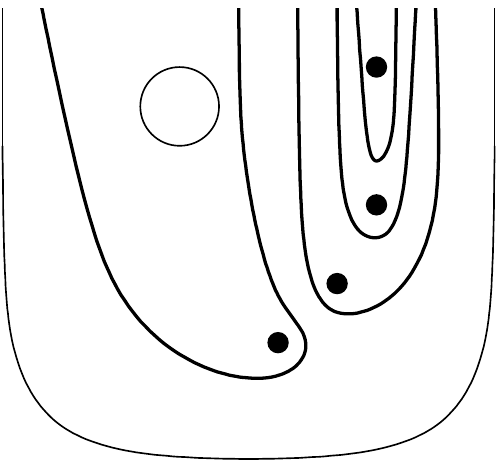}}\hspace{0.1in}\scalebox{0.5}{\includegraphics{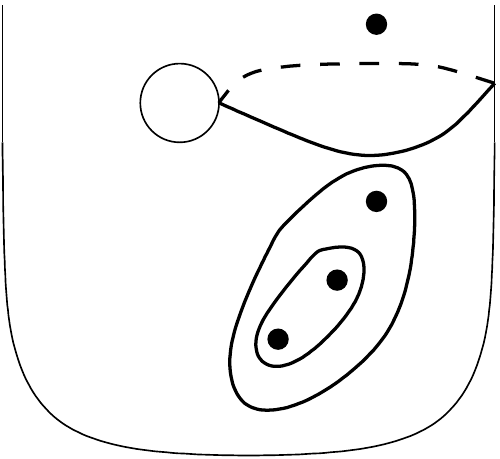}}\end{center}
\caption{Additional curves in families $P$  and $Q$ for $4$ punctures}\label{punct}
\end{figure}
By Lemma \ref{r.index1}, the index of $\bA$ in $\BZ^{6g-6+3p}$ is also $2^{4g-5+2p}$. Hence $\bA$ is equal to the  $\BZ$-span of $\nu(\cP), \nu(\cQ)$ and $\Ad$.\\
To prove Lemma \ref{r.p1} consider  cases when $p>0$ and $p=0$.

(a) Suppose $p>0$. Use the ideal triangulation and the two pants decompositions of Lemma \ref{r.full1}.
When  $p>0$, by the definition in Equation \eqref{eq.defdeg1}, one has $\deg(\al) =\nu(\al)$ for all $\al\in \cS$. Hence
$$ \overline {\deg(\Cd^*)} = \bA, \   \overline { \deg(\BC[\cP]^*)}=  \overline { \nu(\cP)}  , \     \overline {\deg(\BC[\cQ]^*)}=  \overline { \nu(\cQ) }.$$
Thus the right hand side of \eqref{eq.full6} is the $\BZ$-span of $\Ad, \nu(\cP), \nu(\cQ)$, which by Lemma \ref{r.full1} is equal to $\bA$.

(b) Suppose $p=0$. Let $\Sigma$ be a compact planar surface with $g+1$ boundary components, then $\oS= \Sigma\setminus \partial \Sigma$ is a finite type surface of type $F_{0,g}$. Let $(a_i)_{i=1}^{3g-3}$, $\cQ_1$, $\cQ_2$ be respectively the ideal triangulation, the pants decompositions $\cP$ and $\cQ$ of Lemma \ref{r.full1} for the surface $\oS$. Let $\cD\subset \oS$ be the trivalent graph dual to the system $(a_i)_{i=1}^{3g-3}$. We can assume that the topological closure $\bar a_i$ of $a_i$ in $\Sigma$ is a proper embedding of $[0,1]$ into $\Sigma$ and that the $6g-6$ endpoints of all $3g-3$ arcs $\bar a_i$ are distinct. Take another copy $\Sigma'$ of $\Sigma$ and assume that $\varphi: \Sigma \to \Sigma'$ is a diffeomorphism. Let $F$ be the result of gluing $\Sigma$ with $\Sigma'$ along the boundary by the identification $x \equiv \varphi(x)$ for every $x\in \partial \Sigma$. Let $\cP= (P_1,\dots, P_{3g-3})$ where $P_i= \bar a_i \cup \varphi(\bar a_i)$ and $\cQ=(Q_1,\dots, Q_{3g-3})$ be the collection of components of $\cQ_1$, $\partial \Sigma$, and $\varphi(\cQ_2)$, in some order. We claim that the coordinate datum $(\cP, \cD)$ and the two pants decompositions $\cP, \cQ$ satisfy \eqref{eq.full6}. Note that we can take $\Sigma=N(\cD)$, and $\Omega = \partial \Sigma$.

The surface $\oS$ has genus $g'=0$ and puncture number $p'=g+1$. 
It follows from  \eqref{eq.index} that
 \be \label{eq.p9}
 \det\left[  I(P_i, Q_j)_{i,j=1}^{3g-3}   \right] = 2^{4g' -5 + 2p'}= 2^{2g-3}. 
 \ee
 Since $F$ is closed, $\deg(\al)= \eta(\al)$ for $\al\in \cS$ by Definition \eqref{eq.defdeg2}. Hence  
 \be
\overline  {\deg(\BC[\cP]^*)} =  \overline{ \eta(\cP  )}, \quad \overline  {\deg(\BC[\cQ]^*)} =  \overline{ \eta(\cQ)  }.
 \ee
 Thus, to prove the lemma we need to show that 
 \be \overline{ \eta(\cP  )} + \overline{ \eta(\cQ  )}= \bA.
 \label{eq.full4}
 \ee
 Since the left hand side is a subgroup of the right hand side, we only need to show that they have the same index in $\BZ^{6g-6}$.
 
 Let $\vec 0\in \BZ^{3g-3}$ be the 0 vector and $\bode_i \in \BZ^{3g-3}$ is the vector whose entries are all 0 except for the $i$-th one which is 1. By Proposition \ref{r.aden}(e), we have $\eta(P_i)= (\vec 0, -\bode_i)$. Hence $
 \overline{ \eta(\cP  )} = \{\vec 0\} \oplus \BZ^{3g-3}
$. 
 It follows that index of  $ \overline{ \eta(\cQ  )} + \overline{ \eta(\cP  )}$ in $\BZ^{6g-6}$ is equal to the index of  $\left[\overline{ \eta(\cQ  )}\right]_1$ in $\BZ^{3g-3}$, where $\left[\overline{ \eta(\cQ  )}\right]_1 $ is the $\BZ$-span of $(3g-3)$-tuples which are  the first $3g-3$ coordinates of $\eta(Q_i)$, $i=1, \dots, 3g-3$.
 Since each $Q_i$ has 0 intersection with $\Omega$, 
 by Proposition \ref{r.aden} one has $\eta(Q_i)= \nu(Q_i)$. The first $3g-3$ coordinates of $\nu(Q_i)$ are given by $I(Q_i, P_j) _{j=1}^{3g-3}$. Hence the index of 
$\left[\overline{ \eta(\cQ  )}\right]_1$ is $\BZ^{3g-3}$ in $\det\left[  I(P_i, Q_j)_{i,j=1}^{3g-3}  \right]
$, which is equal to $2^{2g-3}$ by \eqref{eq.p9}.
 
Thus the index of $ \overline{ \eta(\cP  )} + \overline{ \eta(\cQ  )}$ in $\BZ^{6g-6}$ is $2^{2g-3}$, equal to  the index of $\bA$ in $\BZ^{6g-6}$, by Lemma \ref{r.index2}. Hence we have \eqref{eq.full4}, which proves the lemma.
\end{proof}

\subsection{Proof of Theorem \ref{r.decomp}} Let $X$ be the image of $\psi$ defined in \eqref{eq.psi} and $X^*= X \setminus \{0\}$. Choose the coordinate datum and the pants decompositions $\cP$, $\cQ$ as in Lemma \ref{r.p1}. Let 
$$ \cB: = \deg(\Cd^*) + \deg(\BC[\cP]^*) + \deg (\BC[\cQ]^*).$$
Note that $\cB$ is a submonoid of $\bA$, and Lemma \ref{r.p1} implies that $\overline {\cB}= \bA$. By Lemma \ref{r.surj} the natural map $\phi: \cB \to \overline {\cB}/\bAz = \bA /\bAz=\Rz$ is surjective.

Let $Y= \{ y_1 y_2 y_3 \mid y_1 \in \Cd^*, y_2 \in \CP^*, y_3 \in \CQ^*\}$, then
$ \deg(Y)=\cB.$
Since $Y \subset X$, we have the following commutative diagram
$$
\begin{tikzcd}
Y  \arrow[twoheadrightarrow,"\deg"]{r}  \arrow[hookrightarrow] {d}& \cB  \arrow[twoheadrightarrow, "\phi"]{d} \\
X  \arrow[rightarrow,"\dz"]{r} & \Rz
\end{tikzcd}
$$
Since  $\deg$ and $\phi$ are surjective, $\dz$ is also surjective. This means $\dz(X)= \Rz$. By Corollary \ref{r.dim=} we have $X= \tKF$. Thus $\psi$ is surjective. Since the dimension over $\tZF$ of the domain and the codomain of $\psi$ are the same,  $\psi$ is a $\tZF$-linear isomorphism.\qed

\no{\subsection{General division algebra}
We also have a more general splitting conjecture
\begin{conjecture}  Let $D$ be a division algebra which is finite dimensional over its center $Z$. Suppose $C_1 \cap C_2=Z$, where $C_1$ and $C_2$ are maximal commutative algebras of $D$, then the map given by sending $c_1\otimes c_2$ to $c_1c_2$
  \begin{equation} C_1\otimes_{Z}C_2\rightarrow D \end{equation}
  is an isomorphism of $Z$-modules. \end{conjecture}
 
 This is equivalent to the statement that if $C_1$ and $C_2$ are maximal commutative subalgebras of $D$ that intersect in $Z$ then $C_1C_2$ is a subalgebra of $D$.
\blue{Do we need this? If this true, it is probably known. If not, it is  very difficult. It is not likely  that we can make a good conjecture in the old theory of division algebras.}}

\no{\appendix
\section{Proof of Lemma \ref{r.plDT}}
\begin{proof}
Let $P_{i+ 3g-3} = K_i'$ in \cite{FLP}, $ P_{i+6g-6}= K''_i$ in \cite{FLP}, and  our $P_i$ be $K_i$  in \cite{FLP}. The intersection number of $\al$ with $P_i$ is a rationally piecewise linear function of the DT $9g-9$ coordinates of \cite{FLP}, which are rationally piecewise linear functions of our DT coordinates.
\blue {Write up the details. the Lemma is known and is written is an unpublished note of D. Thurston.}
\end{proof}}


\begin{thebibliography}{0000}
\bibitem{AF1}   Abdiel, Nelson; Frohman, Charles, {\em Frobenius algebras derived from the Kauffman bracket skein algebra}, J. Knot Theory Ramifications {\bf } 25 (2016), no. 4, 1650016, 25 pp. 
\bibitem{AF} Abdiel, Nelson; Frohman, Charles, {\em The Localized Skein Algebra is Frobenius}, Algebr. Geom. Topol. {\bf 17} (2017), no. 6, 3341-3373. 	
\bibitem{BW2} Bonahon, Francis; Wong, Helen, {\em Representations of the Kauffman bracket skein algebra I: invariants and miraculous cancellations}, Invent. Math. {\bf 204} (2016), no. 1, 195-243. 
\bibitem{BW4} Bonahon, Francis; Wong, Helen {\em Representations of the Kauffman bracket skein algebra II: Punctured surfaces} Algebr. Geom. Topol. {\bf 17} (2017), no. 6, 3399-3434. 
  \bibitem{Cohn} Cohn, P. M., {\em  Algebra} Vol. 3. Second edition. John Wiley \& Sons, Ltd., Chichester, 1991. xii+474 pp. ISBN: 0-471-92840-2.
\bibitem{CM}  Charles, Laurent; March\'{e}, Julien, {\em  Multicurves and regular functions on the representation variety of a surface in $SU(2)$}, Comment. Math. Helv. {\bf 87} (2012), no. 2, 409-431.
\bibitem{FG} 	Frohman, Charles; Gelca, Razvan, {\em Skein modules and the noncommutative torus}, Trans. Amer. Math. Soc. {\bf 352 }(2000), no. 10, 4877-4888.
\bibitem{FK}  Frohman, Charles; Kania-Bartoszynska, Joanna,  {\em The structure of the Kauffman bracket skein algebra at roots of unity}, Math. Z. {\bf 289} (2018), no. 3-4, 889-920. 
\bibitem{FKL} Frohman, Charles; Kania-Bartoszynska, Joanna;  L\^{e}, Thang, {\em Unicity for Representations of the Kauffman Bracket Skein Algebra},  To appear, Inventiones Mathematicae, https://doi.org/10.1007/s00222-018-0833-x 
\bibitem{FLP} Fathi, Albert; Laudenbach, François; Poénaru, Valentin, {\em Thurston's work on surfaces}.
 Mathematical Notes, {\bf 48},  Princeton University Press,  2012.
  \bibitem{FM}  Farb, Benson; Margalit, Dan, {\em A primer on mapping class groups}, Princeton Mathematical Series, {\bf 49} Princeton University Press, Princeton, NJ, 2012. xiv+472 pp. ISBN: 978-0-691-14794-9.
\bibitem{lam} Lam, T. Y., {\em  A first course in noncommutative rings}, Second edition. Graduate Texts in Mathematics, {\bf 131} Springer-Verlag, New York, 2001. xx+385 pp. ISBN: 0-387-95183-0.
\bibitem{Iv}  Ivanov, Nikolai V., {\em Subgroups of Teichm\"{u}ller modular groups}, Translated from the Russian by E. J. F. Primrose and revised by the author. Translations of Mathematical Monographs, {\bf 115} American Mathematical Society, Providence, RI, 1992. xii+127 pp. ISBN: 0-8218-4594-2.
 
 \bibitem{Le0}   L\^{e}, Thang, {\em On Kauffman Bracket Skein Modules at Root of Unity}, Algebraic \& Geometric Topology,  {\bf 15} (2015), no. 2, 1093--1117.
  \bibitem{Le}  L\^{e}, Thang, {\em On positivity of Kauffman bracket skein algebras of surfaces}, International Mathematics Research Notices, 2018, no. 5, 1314--1328. 
    \bibitem{Le1} Le, Thang, {\em Triangular decomposition of skein algebras}, 	Quantum Topology, {\bf 9} (2018), 591--632.
   \bibitem {Le_Paprocki}  L\^{e}, Thang;  Paprocki, Jonathan, {\em On Kauffman bracket skein modules of marked 3-manifolds and the Chebyshev-Frobenius homomorphism},  	arXiv:1804.09303 [math.GT]
\bibitem{Luo1} 	Luo, Feng, {\em Some applications of a multiplicative structure on simple loops in surfaces. Knots, braids, and mapping class groups - papers dedicated to Joan S. Birman (New York, 1998)}, 123--129, AMS/IP Stud. Adv. Math., 24, Amer. Math. Soc., Providence, RI, 2001.
\bibitem{Luo2} 	Luo, Feng, {\em  Simple loops on surfaces and their intersection numbers}, J. Differential Geom. {\bf 85} (2010), no. 1, 73-115.
  \bibitem{LS}  Luo, Feng, Stong, Richard, {\em Dehn-Thurston Coordinates for Curves on Surfaces}, Comm. Anal. Geom., {\bf 12} (2004), no. 1, 1-41.
  \bibitem{M}  Muller, Greg, {\em Skein and cluster algebras of marked surfaces}, Quantum Topol. {\bf 7 }(2016), no. 3, 435-503.
  \bibitem{Mc}    McCarthy, Paul J., {\em Algebraic extensions of fields.}, Corrected reprint of the second edition. Dover Publications, Inc., New York, 1991. x+166 pp. ISBN: 0-486-66651-4

\bibitem{Penner} Penner, Robert, {\em The action of the mapping class group on curves in surfaces}, Enseign. Math. (2) {\bf 30} (1984), no. 1-2, 39-55.
\bibitem{PS} 	Przytycki, J\'{o}zef H.; Sikora, Adam S., {\em On skein algebras and $Sl_2(C)$-character varieties}, Topology {\bf 39} (2000), no. 1, 115-148.
\bibitem{PS1} 	Przytycki, J\'{o}zef H.; Sikora, Adam S., {\em Skein Algebras of Surfaces}, Trans. Amer. Math. Soc. {\bf 371} (2019), no. 2, 1309-1332.
 \bibitem{SW} 	Sikora, Adam S.; Westbury, Bruce W., {\em  Confluence theory for graphs}, Algebr. Geom. Topol. {\bf 7}7 (2007), 439-478.
 
 \bibitem{Th} 	Thurston, Dylan Paul, {\em  Positive basis for surface skein algebras}, Proc. Natl. Acad. Sci. USA {\bf 111} (2014), no. 27, 9725-9732.
  \bibitem{Thurston} 	Thurston, Dylan Paul, {\em Geometric intersection of curves on surfaces}, preprint, available at \url{http://pages.iu.edu/~dpthurst/DehnCoordinates.pdf}.
 \bibitem{Turaev}  V. Turaev,
{\em Conway and Kauffman modules of a solid torus},  Zap. Nauchn. Sem. Leningrad. Otdel. Mat. Inst. Steklov. (LOMI) 167 (1988), Issled. Topol. 6, 79--89, 190; translation in
J. Soviet Math. 52 (1990), no. 1, 2799--2805.
 \end{thebibliography}
 \end{document}